\documentclass[a4paper,11pt]{amsart}
\usepackage{fancyhdr}
\usepackage{amsmath}
\usepackage{dsfont}
\usepackage{hyperref}
\usepackage[mathscr]{eucal}
\usepackage[cp1251]{inputenc}
\usepackage[english]{babel}
\usepackage{cite,enumerate,float,indentfirst}
\usepackage{graphicx}
\usepackage{xcolor}
\usepackage{latexsym,a4,mathrsfs,amsthm,amsmath,amssymb,url}
\usepackage{amsfonts}
\usepackage{amssymb}
\usepackage{epstopdf}

\numberwithin{equation}{section}
\newtheorem{theorem}{Theorem}[section]
\newtheorem{lemma}[theorem]{Lemma}
\newtheorem{statement}[theorem]{Statement}
\newtheorem{definition}[theorem]{Definition}

\theoremstyle{remark}
\newtheorem*{remark}{Remark}

\newcommand{\R}{\mathbb{R}}
\newcommand{\C}{\mathbb{C}}

\newcommand{\Cond}{{\bf (C0)}}
\newcommand{\UI}{{\bf (UI)}}
\newcommand{\Lind}{{\bf (L)}}

\newcommand{\X}{{\bf X}}
\newcommand{\A}{{\bf A}}

\newcommand{\M}{{\bf M}}
\newcommand{\B}{{\bf B}}
\newcommand{\Q}{{\bf Q}}
\newcommand{\EE}{{\bf E}}
\newcommand{\F}{{\bf F}}
\newcommand{\G}{{\bf G}}
\newcommand{\HH}{{\bf H}}
\newcommand{\OO}{{\bf O}}
\newcommand{\OP}{{\bf P}}
\newcommand{\I}{{\bf I}}

\newcommand{\W}{{\bf W}}

\DeclareMathOperator{\Tr}{Tr}

\DeclareMathOperator{\dist}{dist}

\DeclareMathOperator{\E}{\mathbb{E}}
\DeclareMathOperator{\Pb}{\mathbb{P}}

\DeclareMathOperator{\supp}{supp}
\DeclareMathOperator{\incomp}{\mathcal {IC}}

\begin{document}

\vspace{1in}

\title[On minimal singular values]{\bf On minimal singular values of random matrices with correlated entries}


\author[F. G{\"o}tze]{F. G{\"o}tze}
\address{F. G{\"o}tze\\
 Faculty of Mathematics\\
 Bielefeld University \\
 Bielefeld, Germany
}
\email{goetze@math.uni-bielefeld.de}

\author[A. Naumov]{A. Naumov}
\address{A. Naumov\\
 Faculty of Mathematics\\
 Bielefeld University \\
 Bielefeld, Germany \\
 and Faculty of Computational Mathematics and Cybernetics\\
 Moscow State University \\
 Moscow, Russia
 }
\email{naumovne@gmail.com, anaumov@math.uni-bielefeld.de}

\author[A. Tikhomirov]{A. Tikhomirov}
\address{A. Tikhomirov\\
 Department of Mathematics\\
 Komi Research Center of Ural Branch of RAS \\
 Syktyvkar, Russia
 }
\email{tichomir@math.uni-bielefeld.de}
\thanks{All authors are supported by CRC 701 ``Spectral Structures and Topological
Methods in Mathematics'', Bielefeld. A.~Tikhomirov are partially supported by RFBR, grant  N 11-01-00310-a and  grant N 11-01-12104-ofi-m-2011.
A.~Naumov has been supported by the German Research Foundation (DFG) through the International Research Training Group IRTG 1132 and in part by the Simons Foundation.}


\date{\today}

\begin{abstract}
Let $\mathbf X$ be a random matrix whose pairs of entries $X_{jk}$ and $X_{kj}$  are correlated and
vectors $ (X_{jk},X_{kj})$, for $1\le j<k\le n$,  are mutually independent. Assume that the diagonal entries are independent from off-diagonal entries as well. We assume that $\E X_{jk}=0$, $\E X_{jk}^2=1$, for any $j,k=1,\ldots,n$  and  $\E X_{jk}X_{kj}=\rho$ for $1\le j<k\le n$. Let $\mathbf M_n$ be a non-random $n\times n$ matrix with $\|\mathbf M_n\|\le Kn^Q$, for some positive constants $K>0$ and $Q\ge 0$. Let $s_n(\mathbf X+\mathbf M_n)$ denote the least singular value of the matrix $\mathbf X+\mathbf M_n$. It is shown that there exist positive constants $A$ and $B$ depending on $K,Q,\rho$ only such that
\begin{equation}\notag
 \Pb(s_n(\mathbf X+\mathbf M_n)\le n^{-A})\le n^{-B}.
\end{equation}
As an application of this result we prove the elliptic law
for this class of matrices with non identically distributed
correlated entries.
\end{abstract}

\maketitle



\section{Introduction}
Let $\mathbf M_n$ be an $n\times n$ matrix such that $\|\mathbf M_n\|\le Kn^Q$ with some positive constant $K>0$ and a
non-negative constant  $Q\ge 0$. Consider an $n\times n$ matrix $\mathbf W=\mathbf X+\mathbf M_n$, where $\mathbf X$ denotes a random matrix with real entries
$X_{j,k}$, $1\le j,k\le n$, satisfying the following conditions \Cond:\\
a) random vectors $(X_{jk}, X_{kj})$ are mutually independent for $1 \le j < k \le n$;\\
b) for any $j, k = 1, ... , n$
$$
\E X_{j k} = 0 \text{ and } \E X_{j k}^2 = 1;
$$
c) for any $1 \le j < k \le n$
$$
\E ( X_{j k} X_{k j} ) = \rho, |\rho| \le 1;
$$
In the case $\rho=1$ $X$ is a.s. a symmetric matrix, if $\rho=0$ and all random variables are Gaussian it is from the Ginibre ensemble.

We say that the entries $X_{j,k}$, $1\le j,k\le n$, of the matrix $\X$ satisfy condition $\UI$ if the squares of $X_{jk}$'s are uniformly integrable , i.e.
\begin{equation}\label{unif0}
\max_{j,k}\E|X_{jk}|^2\mathbb I{\{|X_{jk}|>M\}} \rightarrow 0 \quad \text{as}\quad M \rightarrow \infty.
\end{equation}
Let us denote the least singular value of the perturbed matrix $\mathbf W$ by $s_n:=s_n(\mathbf W)$.
The main result of this note is
\begin{theorem} \label{th:main}
Assume that $X_{jk}$, $1\le j,k\le n$, satisfy the conditions \Cond and \UI.  Let $\X=\{X_{jk}\}$ denote a $n\times n$ random matrix with entries $X_{jk}$ and let  $\mathbf M_n$ denote a non-random matrix with $\|\mathbf M_n\|\le Kn^Q=:K_n$ for some $K>0$  and $Q\ge0$. Then there exist constants $ C, A,B >0$ depending on  $K,Q$ such that
\begin{align}
\Pb (s_n\le n^{-B } )\le Cn^{-A},
\end{align}
\end{theorem}

\begin{remark}
 Under the conditions of Theorem \ref{th:main} we have, for $\gamma \ge\max\{Q,1\}+\frac12$
\begin{equation}
\Pb (\|\X+\M_n\|>n^{\gamma}\}
 \le \frac {\E\|\mathbf X\|_2^2+\|\mathbf M_n\|^2}{n^{2\gamma}}\le \frac{C}n.
\end{equation}
\end{remark}

A similar result  to Theorem \ref{th:main} was obtained by Tao and Vu in~\cite{TaoVu2010} and by G{\"o}tze and Tikhomirov in~\cite{GotTikh2010} for non-symmetric random matrices (that is $\rho=0$ and $X_{jk}$ and $X_{kj}$ are independent).
Under the condition that $X_{jk}$, for $j,k=1,\ldots, n$ are i.i.d. random variables with subgaussian distribution, Rudelson and Vershynin in~\cite{RudVesh2008} obtained an optimal
 bound for $s_n(\bold X)$ (without non-random shift). In the case of symmetric matrices
($\rho=1$) with i.i.d. entries which have a subgaussian distribution Vershynin proved
\begin{equation}
 \Pb (s_n(\X-z \I)\le \varepsilon n^{-1/2} \text{ and }\|\X\|\le K\sqrt n\}
 \le C\varepsilon^{\frac19}+2\text{\rm e}^{-n^c}.
\end{equation}
The result of Theorem~\ref{th:main} in the case of symmetric matrices
($\M_n$ symmetric as well) was proved by H. Nguyen in~\cite{Ngu2010} assuming a so-called anti-concentration condition. For the i.i.d. case. H. Nguyen used in his proof techniques which are very different from those used by Vershynin.
Recently Nguyen and O'Rourke in~\cite{ORourkeNguyen2012} have proved the result of Theorem~\ref{th:main} for i.i.d. r.v.'s, assuming a finite second order moment. It seems though that their proof is rather involved. In the present note we consider the non-i.i.d. case.
Our proof is short and based on the approach by Rudelson and Vershynin (see \cite{RudVesh2008}) which
divides the unit sphere into two classes of compressible and incompressible vectors.

Throughout this paper we assume that all random variables are defined on a common probability
space $(\Omega, \mathcal{F}, \Pb)$. By $C$ (with an index or without it) we shall denote generic absolute constants, whereas $C(\,\cdot\,,\,\cdot\,)$ will denote positive constants depending on various arguments. For any matrix $\mathbf A$ we shall denote by $\|\mathbf A\|_2$ the Frobenius norm of the matrix $\mathbf A$ ($\|\mathbf A\|_2^2=\Tr\mathbf A\mathbf A^*$) and by $\|\mathbf A\|$ we shall denote the operator norm of the matrix $\mathbf A$ ($\|\mathbf A\| =\sup_{\mathbf x: \|\mathbf x\|=1}\|\mathbf A\mathbf x\|$).

\section{Proof of the main result}
 The proof is similar to the proof of Theorem 4.1 in \cite{GotTikh2010}.  We use ideas of
 Rudelson and Veshynin
\cite{RudVesh2008} to classify with high probability the vectors $\mathbf x$ in the
$(n-1)$-dimensional unit sphere $\mathcal S^{(n-1)}$ such that
$\|\W x\|_2$ is extremely small into two classes, called compressible
and incompressible vectors.
Note that
\begin{equation}
 s_n=\inf_{\mathbf x\in\mathcal S^{(n-1)}}\|\W \mathbf x\|_2.
\end{equation}
First we note that without loss of generality we may assume that the matrix $\mathbf W$ and
all its principal minors are invertible.
Otherwise we may consider the matrix $\mathbf W+\exp\{-n\}r\mathbf I$ where $r$ is a random
variable  which is uniformly
distributed  on the unit interval and independent
of the matrix $\mathbf W$. For such a matrix we have
\begin{equation}
 \Pr\{\det(\mathbf W+\exp\{-n\}r\mathbf I)=0\}
=\E\Pr\{\prod_{j=1}^n(\lambda_j+r\exp\{-n\})=0\Big|\mathbf W\}=0.
\end{equation}
We denote here by $\lambda_1,\ldots,\lambda_n$ the eigenvalues of the matrix $\mathbf W$.
Moreover, the last relation holds for any principal minor of the matrix $\mathbf W$.
We  also have the inequality
\begin{equation}
 s_n(\mathbf W)\ge s_n(\mathbf W+\exp\{-n\}r\mathbf I)-\exp\{-n\}.
\end{equation}
Note that the matrix $\mathbf X+\exp\{-n\}r\mathbf I$ satisfies the conditions~\Cond.

We start now from the following lemma.
\begin{lemma}\label{local}
Let $\mathbf x^T=(x_1,\ldots,x_n)$ be a fixed unit vector and $\mathbf X$ be a matrix as
in Theorem \ref{th:main}.
Then there exist some positive absolute constants $c_0$ and $\tau_0$ such that for any
$0<\tau\le\tau_0$ and any vector $\mathbf u^T=(u_1,\ldots,u_n)$
\begin{equation}
 \Pb (||\W \mathbf x-\mathbf u ||_2/||\mathbf x||_2 \le\tau\sqrt n) \le \exp\{-c_0n\}.
\end{equation}
\end{lemma}
\begin{proof}Let $n_0=[n/2]$ and $n_1=n-n_0$.
Represent the matrix $\mathbf X$ in the form
\begin{equation}
 \mathbf X=\begin{bmatrix}&\mathbf A&\mathbf B\\&\mathbf C&\mathbf D\end{bmatrix}
\end{equation}
and
\begin{equation}
 \mathbf M_n=\begin{bmatrix}&\mathbf M_n^{(11)}&\mathbf M_n^{(12)}\\&\mathbf M_n^{(21)}
 &\mathbf M_n^{(22)}\end{bmatrix},
\end{equation}
where $\mathbf B$, $\mathbf M_n^{(12)}$ are matrices of dimension $n_0\times n_1$,
$\mathbf C$, $\mathbf M_{n}^{(21)}$ are
$n_1\times n_0$ matrices,
 $\mathbf A$, $\mathbf M_n^{(11)}$ are $n_0\times n_0$ matrices, and  $\mathbf D$,
 $\mathbf M_n^{(22)}$ are $n_1\times n_1$ matrices.
Let $\mathbf x^T=(\mathbf x_0,\mathbf x_1)^T\in\mathcal S^{(n-1)}$ where
$\mathbf x_i\in\mathbb R^{n_i}$, $i=0,1$. Split the vector
$\mathbf u^T$ as
$\mathbf u^T=(\mathbf u_0,\mathbf u_1)^T$.
Using these notations we have
\begin{align}\notag
 \|\W \mathbf x-\mathbf u\|_2^2&=
 \|(\mathbf A\mathbf x_0+\mathbf B\mathbf x_1+\mathbf M_n^{(11)}\mathbf x_0+
\mathbf M_n^{(12)}\mathbf x_1-\mathbf u_0\|_2^2\notag\\&
+\|(\mathbf C\mathbf x_0+\mathbf D\mathbf x_1+\mathbf M_n^{(21)}\mathbf x_0+
\mathbf M_n^{(22)}\mathbf x_1-\mathbf u_1\|_2^2
\end{align}
Note that  $\max\{||\mathbf x_0||_2,||\mathbf x_1||_2\}\ge\frac{\|\mathbf x\|_2}{\sqrt 2}$.
 Without loss of generality we may assume that
$\|\mathbf x_1\|_2 \ge \frac{\|\mathbf x\|_2}{\sqrt 2}$. We may write then
\begin{equation}\notag
 \|\W \mathbf x-\mathbf u\|_2^2\ge \|\mathbf A\mathbf x_0
 +\mathbf B\mathbf x_1+\mathbf M_n^{(11)}\mathbf x_0+
\mathbf M_n^{(12)}\mathbf x_1-\mathbf u_0\|_2^2.
\end{equation}
Denote by $\mathbf y^{(0)}=\mathbf A\mathbf x_0+\mathbf M_n^{(11)}\mathbf x_0+
\mathbf M_n^{(12)}\mathbf x_1-\mathbf u_0$. Note that $\mathbf B\mathbf x_1$ and
$\mathbf y^{(0)}$ are independent.
Furthermore, all entries of matrix $\mathbf B$ are independent and $\E[\mathbf B]_{j,k}=0$,
$\E[\mathbf B]_{jk}^2=1$.
We have
\begin{equation}\notag
 \|\mathbf B\mathbf x_1+\mathbf y_0\|_2^2/\|\mathbf x_1\|_2^2=\sum_{j=1}^{n_0}
 (\sum_{k=1}^{n_1}[\mathbf B]_{jk}\widetilde x_{1k}+\widetilde y^{(0)}_j)^2
=:\sum_{j=1}^{n_0}|\zeta_j+\widetilde y^{(0)}_j|^2,
\end{equation}
where $\zeta_j=\sum_{k=n_0+1}^{n} X_{jk}\widetilde x_{1k}$ and $\widetilde{\mathbf x}_1
=\mathbf x_1/\|\mathbf x_1\|_2=(\widetilde x_{1 n_0+1},\ldots,\widetilde x_{1n})^T$,
${\widetilde{\mathbf y}}^{(0)}=\mathbf y^{(0)}/\|\mathbf x_1\|_2=
({\widetilde y}^{(0)}_{1},\ldots,{\widetilde y}^{(0)}_{n_0})^T$. The remaining part of
the proof is similar to the proof of Lemma 4.1 in \cite{GotTikh2010}.
By Chebyshev's inequality
\begin{align}\label{chebysh}
\Pb (||\W \mathbf x-\mathbf u ||_2/||\mathbf x||_2\le\tau\sqrt n )&
 \le \Pb(||\B\mathbf x_1+\mathbf y^{(0)}||_2^2/||\mathbf x_1||_2^2\le 2 n\tau^2)
 \notag\\&\le
\exp\{nt^2\tau^2\}\prod_{j=1}^{n_0}\E\exp\{-\frac{t^2}2(\zeta_j
+{\widetilde y}^{(0)}_j)^2\}.
\end{align}
Using $e^{-t^2/2}=\E\exp\{it\xi\}$, where $\xi$ is a standard Gaussian random variable, we obtain
\begin{align}\label{product}
\Pb (\sum_{j=1}^{n_0}(\zeta_j+{\widetilde y}^{(0)}_j)^2\le 2 n\tau^2 )) &\le \exp\{nt^2\tau^2\}\prod_{j=1}^{n_0}\E_{\xi_j}\exp\{itn\widetilde y_j^{(0)}\} \notag \\&
\times \prod_{k=n_0+1}^{n}\E_{X_{jk}}\exp\{it\xi_j\widetilde x_{1k}X_{jk}\},
\end{align}
where $\xi_j$, $j=1,\ldots, n_0$ denote i.i.d. standard Gaussian r.v.'s and $\E_Z$
denotes expectation with respect to
$Z$ conditional on all other r.v.'s.
Take $\alpha=\Pr\{|\xi_1|\le C_1\}$ for some absolute positive constant $C_1$ which
will be chosen later. Then it follows from (\ref{product})
\begin{align}\label{chebysh1}
\Pb (\sum_{j=1}^{n_0}(\zeta_j+y^{(0)}_j)^2&\le 2 n\tau^2)
\le\exp\{nt^2\tau^2\}\notag\\& \times\prod_{j=1}^{n_0}\left(\alpha\E_{\xi_j}
\left|\prod_{k=n_0+1}^{n}\E_{X_{jk}}
\left\{\exp\{it\xi_j \widetilde x_{1k}X_{jk}\}\Big||\xi_j|\le C_1\right\}\right|+1-\alpha\right).
\end{align}
Furthermore, note that for any r.v. $\xi$,
\begin{equation}\notag
 |\E {\rm e}^{it\xi}|\le \exp\{-(1-|\E {\rm e}^{it\xi}|^2)/2\}.
\end{equation}
This implies
\begin{equation}\label{411}
\left| \E_{X_{jk}}
\left\{\exp\{it\xi_j \widetilde x_{1k} X_{jk}\}\Big||\xi_j|\le C_1\right\}\right|\le
\exp\{-(1-|f_{jk}(t \widetilde x_{1k}\xi_j)|^2)/2\}
\end{equation}
where $f_{jk}(u)=\E\exp\{iuX_{jk}\}$. Assuming (\ref{unif0}), choose a constant $M>0$ such that
\begin{equation}\notag
\sup_{j,k}\E|X_{jk}|^2\mathbb I\{|X_{jk}|>M\}\le \frac12.
\end{equation}
Since $1-\cos x\ge 11/24 \,x^2$, for $|x|\le 1$, conditioning on the event $|\xi_j|\le C_1$
we get for $0<t\le 1/(MC_1)$
\begin{equation}\notag
 1-|f_{jk}(t \widetilde x_{1k}\xi_j)|^2\ge 11/24\, t^2{ \widetilde x_{1k}}^2\E|X_{jk}^{(q)}|^2\mathbb I\{|X_{jk}|\le M\}
\ge 11/48\, t^2{ \widetilde x_{1k}}^2.
\end{equation}
 It follows from (\ref{411}) for $0<t<1/(MC_1)$ and for some constant $c>0$ that
\begin{equation}\notag
\left| \E_{X_{jk}}
\left\{\exp\{it\xi_j \widetilde x_{1k}X_{jk}\}\Big||\xi_j|\le C_1\right\}\right|\le
\exp\{-ct^2{ \widetilde x_{1k}}^2\xi_j^2\}.
\end{equation}
This implies that conditionally on $|\xi_j|\le C_1$ and for $0<t<1/(MC_1)$
\begin{equation}\label{chebysh2}
 \left|\prod_{k=n_0+1}^{n}\E_{X_{jk}}
\left\{\exp\{it\xi_j \widetilde x_{1k}X_{jk}\}\Big||\xi_j|\le C_1\right\}\right|\le \exp\{-ct^2\xi_j^2\}
\end{equation}
Let $\Phi_0(x)=\frac1{\sqrt{2\pi}}\int_0^x\exp\{-u^2/2\}du$. Then
\begin{equation}\notag
 \E_{\xi_j}\Big(\exp\{-ct^2\xi_j^2\}||\xi_j|\le C_1\Big)=\frac1{\sqrt{1+2ct^2}}
 \frac{\Phi_0(C_1\sqrt{1+2ct^2})}{\Phi_0(C_1)}.
\end{equation}
We may choose $C_1$ large enough such that following inequality holds:
\begin{equation}\notag
\E_{\xi_j}\Big(\exp\{-ct^2\xi_j^2\}||\xi_j|\le C_1\Big)\le \exp\{-ct^2/24\},
\end{equation}
for all $0<t\le 1/(MC_1)$. Note that for every $\alpha,x\in[0,1]$ and $\beta\in(0,1)$
the following inequality holds
\begin{equation}\label{chebysh3}
 \alpha x+1-\alpha\le \max\{x^{\beta},\,
 \left(\frac{\beta}{\alpha}\right)^{\frac{\beta}{1-\beta}}\}.
\end{equation}
Combining this inequality with inequalities (\ref{chebysh}), (\ref{chebysh1}),
(\ref{chebysh2}), (\ref{chebysh3}), we get
\begin{equation}\notag
\Pb (\sum_{j=1}^{n_0}(\zeta_j+y^{(0)}_j)^2\le 2\tau^2n )\le
\exp\{n\tau^2t^2\}\left(\exp\{-\beta cnt^2/24\}+\left(\frac{\beta}
{\alpha}\right)^{\frac{\beta}{1-\beta}}\right).
\end{equation}
Without loss of generality we may take $C_1$ such that $\alpha\ge 4/5$ and choose
$\beta=2/5$. Then we obtain
\begin{equation}\notag
\Pb (\sum_{j=1}^{n_0}(\zeta_j+y^{(0)}_j)^2\le 2\tau^2n )\le
\exp\{n\tau^2t^2\}\left(\exp\{-cnt^2/60\}+\left(\frac12\right)^{2n/3}\right).
\end{equation}
We conclude from here that there exists a constants $\tau_0>0$ and $c_0>0$ such that
for every $0<\tau\le \tau_0$
\begin{equation}\notag
\Pb (\sum_{j=1}^{n_0}(\zeta_j+y^{(0)}_j)^2\le 2\tau^2n )\le
\exp\{-c_0n\}.
\end{equation}
Thus Lemma \ref{local} is proved.
\end{proof}

Following Rudelson and Vershynin \cite{RudVesh2008} we shall partition the unit sphere
$\mathcal S^{(n-1)}$ into two sets of so-called
compressible and incompressible vectors.
\begin{definition}
 Let $\delta,r\in(0,1)$. A vector $x\in\mathbb R^n$ is called $\delta$-sparse if
 $|\supp(\mathbf x)|\le\delta n$. A vector
$x\in\mathcal S^{(n-1)}$ is called $(\delta,r)$-compressible $\mathbf x$ if within
Euclidean distance
$r$ from the set of all $\delta$-sparse vectors.
A vector
$x\in\mathcal S^{(n-1)}$ is called $(\delta,r)$-incompressible if it is not $(\delta,r)$
-compressible.
\end{definition}
For any fixed $k=1,\ldots,n$ denote by $\mathcal C_{k,n}$ the set of all sparse vectors
$\mathbf x\in\mathcal S^{(n-1)}$
with $|\supp(\mathbf x)|\le k$.
Let $K_n:=Kn^Q$ and $\mathcal E_K: = \{||\W|| \le K_n\}$. Without loss of generality we may assume that $Q\ge \frac12$.
We prove the following analogue of Proposition 4.6 in \cite{GotTikh2010}.
\begin{lemma}\label{help1}
Let $\X$ be as in Theorem~\ref{th:main}. Assume there exist an absolute constants $c_0>0,\,K\ge 1,\, Q\ge\frac12$, and values $\gamma_n,q_n>0$ such that for any
$\mathbf x$ such that $\frac{\mathbf x}{\|\mathbf x\|_2}\in\mathcal C\subset \mathcal S^{(n-1)}$
and for  any $\mathbf u$
inequality
\begin{equation}\label{cond1}
\Pb (||\W \mathbf x -\mathbf u||_2/||\mathbf x||_2 \le\gamma_n\sqrt n,  \mathcal E_K )\le \exp\{-c_0nq_n\}
\end{equation}
holds.
Then there exists a constant $\delta_0$ depending on $K,\,Q,$ and $c_0$ only such that,
for $k=[ \delta_0n q_n/\ln n]$ ,
\begin{equation}\label{lefthand}
\Pb \left (\inf_{\frac{\mathbf x}{\|\mathbf x\|_2}\in\mathcal C_{k,n}\cap\mathcal C}
 ||\W \mathbf x-\mathbf u ||_2/||\mathbf x||_2 \le \frac{\gamma_n\sqrt n}{2},
\mathcal E_K \right )\le \exp \left \{- \frac{c_0n q_n}{8} \right \}.
\end{equation}
\end{lemma}

\begin{proof}The proof is similar to the proof of Proposition 4.2 in \cite{Veshyn2011}.
Let $\eta>0$ to be chosen later. There exists an $\eta$-net $\mathcal N$
in $\mathcal C_{k,n}\cap\mathcal C$ of cardinality
\begin{equation}\label{netn}
|\mathcal N|\le \left(\frac{3\ln n}{\delta_0\eta q_n }\right)^{2k}
\end{equation}
 (see, e.g.,\cite[Lemma~3.7]{Veshyn2011}). Let $\mathcal E$ denote the event
 in the left hand side of \eqref{lefthand} whose probability we would like to bound.
 Assume that $\mathcal E$ holds. Then there exist vectors
$\mathbf x_0=\frac{\mathbf x}{\|\mathbf x\|_2}\in\mathcal C_{k,n}\cap\mathcal C$ and
$\mathbf u_0=\frac {\mathbf u}{\|\mathbf x\|_2}\in\text{\rm span}(\mathbf u)$ such that
\begin{equation*}
\|\W\mathbf x_0-\mathbf u_0\|_2\le \gamma_n\sqrt n/2.
\end{equation*}
By definition of $\mathcal N$ there exists $y_0\in \mathcal N$ such that
\begin{equation*}
\|\mathbf x_0-\mathbf y_0\|_2\le \eta.
\end{equation*}
On the other hand,
\begin{equation*}
||\W\mathbf y_0||_2\le ||\W || \le K_n.
\end{equation*}
Furthermore,
\begin{align}\label{y_0}
\|\W\mathbf y_0-\mathbf u_0\|_2 \le || \W || ||\mathbf y_0-\mathbf x_0||_2+ ||\W\mathbf x_0-\mathbf u_0||_2 \le K_n\eta+\gamma_n\sqrt n/2.
\end{align}
We choose
\begin{equation*}
\eta=\frac{\gamma_n\sqrt n}{4K_n}.
\end{equation*}
Then we get
\begin{equation*}
\|\mathbf u_0\|_2\le K_n\eta+\gamma_n\sqrt n/2+K_n\le CK_n.
\end{equation*}
We see that
\begin{equation*}
\mathbf u_0\in \text{\rm span}(\mathbf u)\cap CK_n\mathcal C_2^n=:H.
\end{equation*}
Here by $\mathcal C_2^n$ we denote the unit ball in $\mathcal C^n$.
Let $\mathcal M$ be some fixed $(\frac{\gamma_n\sqrt n}{4})$-net on the interval $H$ such that
\begin{equation}\label{netm}
|\mathcal M|\le \frac{4CK_n}{\gamma_n\sqrt n}.
\end{equation}
Let us choose a vector $\mathbf v_0\in\mathcal M$ such that $\|\mathbf y_0-\mathbf v_0
\|\le\frac{\gamma_n\sqrt n}{4}$. It follows from \eqref{y_0} that
\begin{equation*}
\|\W\mathbf y_0-\mathbf v_0\|_2 \le K_n\eta+\gamma_n\sqrt n/2
+\frac{\gamma_n\sqrt n}{4}\le \gamma_n\sqrt n.
\end{equation*}
Summarizing, we have shown that the event $\mathcal E$ implies the existence of vectors
$\mathbf y_0\in\mathcal N$ and $\mathbf v_0\in\mathcal M$ such that
\begin{equation*}
\|\W \mathbf y_0-\mathbf v_0\|_2 \le \gamma_n\sqrt n.
\end{equation*}
Applying condition \eqref{cond1} and estimates \eqref{netm}, \eqref{netn} on the
cardinalities of the nets, we obtain
\begin{equation*}
\Pb (\mathcal E )\le \left(\frac{3 \ln n}{\delta_0\eta q_n}\right)^{2k}
\frac{4CK_n}{\gamma_n\sqrt n}\exp\{-c_0nq_n\}
\end{equation*}
Choosing $k=[\delta_0nq_n/\ln n]$ for some small $\delta_0>0$, we get
\begin{equation*}
\Pb (\mathcal E )\le \exp\{-c_0nq_n/8\}.
\end{equation*}
Thus Lemma \ref{help1} is proved completely.
\end{proof}
Let $\mathcal C(\delta)$ denote the set of all $\delta$-sparse vectors and
$\mathcal C(\delta,r)$ denote the set of $(\delta,r)$-compressible vectors.

\begin{lemma}\label{help2}
Let $\mathbf X$ be a random  matrix as described in Theorem \ref{th:main}.
Assume that there exist an absolute constant $c_0>0$ and values $\gamma_n,q_n>0$ such that
(\ref{cond1})
holds for any $\mathbf x$ such that $\frac{\mathbf x}{\|\mathbf x\|}\in \mathcal C$
and $\mathbf u$.
Then there exist $\delta_1,c_1>0$ that depend on $K,\, Q$ and $c_0$ only, such that
\begin{align}\label{sparse1}
\Pb \left (\inf_{\frac{\mathbf x}{\|\mathbf x\|_2}\in \mathcal C(\delta_1\widehat q_n,r_n)\cap\mathcal C}
||\W \mathbf x-\mathbf u ||_2/||\mathbf x||_2 \le \frac{\gamma_n\sqrt n}{4}, \mathcal E_K \right )  \le \exp\{-c_1nq_n\},
\end{align}
where $\widehat q_n=q_n/\ln n$ and $r_n=\gamma_n\sqrt n/(4K_n)$.

\end{lemma}
\begin{proof}
Choose now $r_n=\gamma_n\sqrt n/4K_n$. Let $V$ be the event on the left hand side of
\eqref{sparse1}. 
Assume that the event $V$ occurs
 for some point $\frac{\mathbf y}{\|\mathbf y\|_2}\in\mathcal C(\delta_1\widehat q_n,r_n).$
Choose a point
$\mathbf x_0\in \mathcal C(\delta_1\widehat q_n)$ such that
$\|\frac{\mathbf y}{\|\mathbf y\|_2}-\mathbf x_0\|_2\le r_n$.
Then
\begin{equation}\notag
||\W\mathbf x_0-\mathbf u ||_2/||\mathbf y||_2 \le\gamma_n\sqrt n/2.
\end{equation}
Put $\mathbf x=\|\mathbf y\|_2\mathbf x_0$. Then, $\frac{\mathbf x}{\|\mathbf x\|_2}=\mathbf x_0
\in \mathcal C(\delta_1\widehat q_n)$ and
\begin{align}
||\W\mathbf x-\mathbf u||_2/||\mathbf x||_2\le\gamma_n\sqrt n/2.
\end{align}
According to Lemma \ref{help1} for any $\delta_1\le \delta_0$ and for any $\tau\le\gamma_n/4$,
we have the following inequality:
\begin{align*}
 \Pb\left (\inf_{\frac{\mathbf x}{\|\mathbf x\|_2} \in\mathcal C(\delta_1\widehat q_n)\cap\mathcal C} ||\W \mathbf x-\mathbf u||_2/ ||\mathbf x||_2 \le\tau\sqrt n, \mathcal E_K \right\} \le \exp\{-c_0nq_n/8\}.
\end{align*}
Thus Lemma \ref{help2} is proved.
\end{proof}
Let $\mathcal {IC}(\delta,r)$ denote the set of $(\delta,r)$-incompressible vectors in
$\mathcal S^{(n-1)}$.
\begin{lemma}\label{incomp1}
 Let $\delta_n,r_n\in(0,1)$. Let $\mathbf X$ be a matrix as described in Theorem \ref{th:main}.
Then there exist some positive constants $c_1$ and $c_2$ and $\delta^{(1)}>0$ depending on
$K$ and $Q$ such that for any $0<\tau<\gamma_n$
\begin{align}\notag
\Pb\left (\inf_{\frac{\mathbf x}{\|\mathbf x\|_2}\in\mathcal C(\delta^{(1)}\widehat q_n) \cap\mathcal {IC}(\delta_n,r_n)} ||\W \mathbf x-\mathbf u||_2/||\mathbf x||_2
 \le\tau\sqrt n, \mathcal E_K \right ) \le \exp\{-c_1 n\ln(n\delta_n)\}
\end{align}
with $\gamma_n=c_2 \left (\frac{\delta_n}{n} \right )^{1/4} r_n$ and $\widehat q_n=\frac{\ln(n\delta_n)}{\ln n}$.
\end{lemma}
\begin{proof}
Introduce representation of matrices $\mathbf X$ and $\mathbf M_n$ as in the proof of Lemma
\ref{local} with $n_0=[n/2]$ and $n_1=n-n_0$
\begin{equation}\notag
 \mathbf X=\begin{bmatrix}&\mathbf A&\mathbf B\\&\mathbf C&\mathbf D\end{bmatrix},\quad
\mathbf M_n=\begin{bmatrix}&\mathbf M_n^{(11)}&\mathbf M_n^{(12)}\\&\mathbf M_n^{(21)}&
\mathbf M_n^{(22)}\end{bmatrix}.
\end{equation}
Let $\frac{\mathbf x}{\|\mathbf x\|_2}\in\mathcal {IC}(\delta_n,r_n)$.
By Lemma~\ref{l,a:incomp vec} there exists a set
$\sigma(\mathbf x)\subset\{1,\ldots,n\}$ of cardinality
$|\sigma(\mathbf x)|\ge \frac12n\delta_n$ such that
\begin{equation}\label{versh10}
 \|P_{\sigma(\mathbf x)}\mathbf x\|_2^2\ge \frac{r_n}2\|\mathbf x\|_2
\end{equation}
and
\begin{equation}\label{versh11}
 \frac{r_n}{\sqrt{2n}}\|\mathbf x\|_2\le |x_k|\le \frac1{\sqrt{n\delta_n/2}}\|\mathbf x\|_2,
 \quad\text{for any}\quad k\in\sigma(\mathbf x).
\end{equation}
By $P_{\sigma(\mathbf x)}$ we denote the projection onto $\mathbb R^{\sigma(x)}$ in
$\mathbb R^n$.
Note that in the representation $\mathbf x=(\mathbf x_0,\mathbf x_1)^T$ at least one the sets $\sigma(\mathbf x_i)$ for the vectors $\mathbf x_i$, $i=0,1$
satisfies $|\sigma(\mathbf x_i)|\ge \frac14n\delta_n$ and
$\frac{r_n}{\sqrt{2n}}\le\frac{|x_{ik}|}{\|\mathbf x\|_2}\le \frac1{\sqrt{n\delta_n/2}}$,
for $k\in\sigma(\mathbf x_i)$.
Assume for definiteness that it holds for $i=1$. Then
\begin{equation}
\frac{r_n\sqrt{\delta_n}}{2\sqrt 2}\|\mathbf x\|_2\le\|\mathbf x_1\|_2 \le \|\mathbf x\|_2.
\end{equation}
Furthermore, we have
\begin{equation}\label{ineq1}
||\W \mathbf x-\mathbf u||_2 / ||\mathbf x||_2\ge \|\mathbf A\mathbf x_0+\mathbf M_n^{(11}\mathbf x_0
+\mathbf M_n^{(12}\mathbf x_1+\mathbf B\mathbf x_1-\mathbf u_0\|_2/\|\mathbf x\|_2.
\end{equation}
Introduce the notation $\mathbf y=(y_1,\ldots,y_{n_0})^T=\mathbf A\mathbf x_0+\mathbf M_n^{(11}
\mathbf x_0+\mathbf M_n^{(12}\mathbf x_1-\mathbf u_0$.
Note that the vectors $\mathbf y$ and $\mathbf B\mathbf x^{(0)}$ are independent and that all
entries of the matrix $\mathbf B$
are mutually independent. We may rewrite inequality (\ref{ineq1}) in the form
\begin{equation}\label{ineq25}
||\W \mathbf x-\mathbf u ||_2 ^2/ ||\mathbf x ||_2^2 \ge ||\mathbf B\mathbf x_1
+\mathbf y ||_2^2/ ||\mathbf x||_2^2 =\sum_{j=1}^{n_0}(\zeta_j+\widetilde y_j)^2,
\end{equation}
where $\zeta_j=\sum_{k=1}^{n_1}X_{j,k+n_0}\widetilde x_{1k}$, and
$\widetilde {\mathbf x}_1=\frac{\mathbf x_1}{\|\mathbf x\|_2}=(\widetilde x_{11},\ldots,
\widetilde x_{1n_1})^T$, $\widetilde{\mathbf y}=\frac{\mathbf y}{\|\mathbf x\|_2}
=(\widetilde y_1,\ldots,\widetilde y_{n_0})^T$.
 Note that $\zeta_j$ are mutually independent for $j=1,\ldots,n_0$.
We introduce now the maximal concentration function of weighted sums of the rows of the matrix
$\mathbf B=(X_{jk})_{j=1,\ldots,n_0;\, k=n_0+1,\ldots,n}$ as,
\begin{equation}\notag
p_{\mathbf x}(\eta)=\max_{j=1,\ldots,n_0}\sup_{u\in\mathbb R}
 \Pb \left (\left|\sum_{k=1}^{n_1} X_{j,k+n_0}\widetilde x_{1k}-u\right|\le\eta\right ).
\end{equation}
From Lemma~\ref{ap:l_conc_funct_2} it follows that
\begin{equation}\label{concentr}
p_{\mathbf x}(\eta) \le c n^{-1/4} \delta_n^{-3/8},
\end{equation}
for all $\eta \le c' \left (\frac{\delta_n}{n}\right )^{1/4} r_n$.
Now we state an analogue of the  tensorization Lemma in \cite[Lemma~4.5]{GotTikh2010}.
\begin{lemma}\label{tenzor}
Let $\zeta_1,\ldots,\zeta_n$ be independent non-negative random variables. Assume that
\begin{equation}\notag
\Pb (\zeta_j\le \lambda)\le b_n
\end{equation}
for some $b_n\in(0,\frac14)$ and $\lambda>0$. Then
\begin{equation}\notag
\Pb \Big (\sum_{j=1}^n\zeta_j^2\le \alpha_0n\lambda^2\Big )\le\exp\{-\frac12n\ln (1/2b_n)\},
\end{equation}
where $\alpha_0=\frac{\sqrt2-1}2$.
\end{lemma}
\begin{proof}[Proof of Lemma~\ref{tenzor}]
For any $\tau>0$ we have
\begin{equation}\notag
\Pb\Big (\sum_{j=1}^n\zeta_j^2\le nt^2\Big )\le\exp\{n\tau\}\prod_{j=1}^n\E\exp\{-\zeta_j^2
 \tau/t^2\}.
\end{equation}
Furthermore,
\begin{align}
\E\exp\{-\tau\zeta_j^2/t^2\}&=\int_0^{\infty}\Pb (\exp\{-\tau\zeta_j^2/t^2\}>s )ds\notag\\&=
\int_0^1\Pb\Big (1/s>\exp\{\tau\zeta_j^2/t^2\Big )ds\notag\\
&\le\int_0^1\Pb (\zeta_j\le\frac t{\sqrt{\tau}}\sqrt{\ln(1/s)})ds\notag\\&
\le\left(\int_0^{\exp\{-\frac{\lambda^2\tau}{t^2}\}}+
\int_{\exp\{-\frac{\lambda^2\tau}{t^2}\}}^1\right)\Pb (\zeta_j\le\frac t{\sqrt{\tau}}
\sqrt{\ln(1/s)})ds\notag\\&\le
b_n+\exp\{-\frac{\lambda^2\tau}{t^2}\}.\notag
\end{align}
Choose now $\tau:=\frac12\ln (1/2b_n)$, $t^2:=\frac{\lambda^2\tau}{\ln (1/b_n)}$. Then we get
\begin{equation}\notag
 \E\exp\{-\tau\zeta_j^2/t^2\}\le 2b_n
\end{equation}
and
\begin{equation}\notag
 \Pb\Big (\sum_{j=1}^n\zeta_j^2\le nt^2\Big )\le\exp\{-\frac12n\ln(1/ 2b_n)\}.
\end{equation}
To conclude the proof of Lemma we note that $\frac{\ln(1/b_n)}{\ln(1/2b_n)}\ge \frac{\sqrt2-1}2$.
Thus Lemma \ref{tenzor} is proved.
\end{proof}
We continue now the proof of Lemma \ref{incomp1}. By the inequality~\eqref{concentr}, we have,
for $\frac{\mathbf x}{\|\mathbf x\|_2}\in\mathcal{IC}(\delta_n,r_n)$,
\begin{equation}\notag
 \Pb (|\zeta_j+y_j|\le \eta ) \le c n^{-1/4} \delta_n^{-3/8},
\end{equation}
for all $\eta \le c'\left (\frac{\delta_n}{n}\right )^{1/4} r_n $.
By Lemma \ref{tenzor} and the inequality (\ref{ineq25}), we have
\begin{equation}\label{ineq30}
 \Pb (||\W\mathbf x-\mathbf y||_2/ ||\mathbf x||_2\le \lambda \sqrt n) \le \exp\{-c_1n\ln(n\delta_n)\}.
\end{equation}
for all $\lambda \le c_2 \left (\frac{\delta_n}{n}\right )^{1/4} r_n $. The last inequality and Lemma \ref{help1} together conclude the proof of Lemma \ref{incomp1}.
\end{proof}

Now we continue with the proof of Theorem \ref{th:main}. Let $\delta_n^{(1)}=\frac{\delta_0}{\ln n}$
and $r_n^{(1)}=\frac{\tau_0 \sqrt n}{4K_n}$.
Consider the set
$$
\mathcal C_0=\mathcal C(\delta_n^{(1)},r_n^{(1)}).
$$
According to Lemma \ref{local} for every $\mathbf x$ we have, for any $0<\tau\le \tau_0$,
\begin{equation}\notag
 \Pb (\|\W \mathbf x-\mathbf y\|_2/\|\mathbf x\|_2\le \tau\sqrt n, \mathcal E_K )\le\exp\{-c_0n\}.
\end{equation}
This inequality and Lemmas \ref{help1} and \ref{help2} together imply, for any $0<\tau\le \tau_0/4$,
\begin{equation}\notag
\Pb (\inf_{\mathbf x\in\mathcal C_0}\|\W \mathbf x-\mathbf y\|_2/\|\mathbf x\|_2\le \tau\sqrt n, \mathcal E_K )\le\exp\{-c_0'n\}.
\end{equation}
By Lemma \ref{incomp1}, the inequality \eqref{ineq30}, for every $\frac{\mathbf x}{\|\mathbf x\|_2}
\in\mathcal {IC}(\delta_n^{(1)},r_n^{(1)})$,
\begin{align}\label{eq: boind for incomp vect}
 \Pb (\|\W \mathbf x-\mathbf y\|_2/\|\mathbf x\|_2\le \tau \sqrt n, \mathcal E_K)\le\exp\{-c_1 n\ln(n \delta_n^{(1)})\}
\end{align}
for all $0 < \tau < c_2 \left (\frac{\delta_n^{(1)}}{n} \right )^{1/4} r_n^{(1)}$.
Furthermore, introduce
$$
\delta_n^{(2)}={\delta_1}\frac{\ln (n\delta_n^{(1)})}{\ln n} \text{ and  }r_n^{(2)}
=c_3\frac{(\delta_n^{(1)} n)^{1/4} r_n^{(1)}}{4K_n}.
$$
Let $\mathcal C_1=\mathcal C(\delta_n^{(2)},r_n^{(2)})\cap\mathcal {IC}(\delta_n^{(1)},r_n^{(1)})$.
The inequality~\eqref{eq: boind for incomp vect} and Lemmas~\ref{incomp1} and~\ref{help2} together imply, for any
$0<\tau\le c_2' \left (\frac{\delta_n^{(1)}}{n} \right )^{1/4} r_n^{(1)}$,
\begin{equation}\label{ineq100}
\Pb (\inf_{\mathbf x\in\mathcal C_1}\|\W \mathbf x-\mathbf y\|_2/\|\mathbf x\|_2 \le \tau\sqrt n, \mathcal E_K ) \le \exp\{-c_1'n\ln(n\delta_n^{(1)})\}.
\end{equation}
Note that there exists an absolute constant $\delta_2>0$ such that $\delta_n^{(2)}\ge\delta_2$.
This implies
\begin{equation}\notag
 \mathcal {IC}(\delta_n^{(2)},r_n^{(2)})\subset\mathcal {IC}(\delta_2,r_n^{(2)})=:\mathcal C_2.
\end{equation}
In what follows we shall bound the quantity
\begin{equation}\notag
 \mathcal P:=\Pb (\inf_{\mathbf x\in\mathcal C_2}\|\W \mathbf x\|_2
 \le \tau\sqrt n,\mathcal E_K).
\end{equation}

We reformulate Lemma 4.9 from \cite{GotTikh2010} (Lemma 3.5 in \cite{RudVesh2008})  for our case.
\begin{lemma}\label{invertdist}Let the matrix $\mathbf X$ be as described in Theorem \ref{th:main}.
 Let $\mathbf X_1,\ldots,\mathbf X_n$ denote the columns of $\W$ and let $\mathcal H_k$
denote the span of all column vectors except the $k$th. Then for every $\eta>0$
\begin{align*}
 \Pb (\inf_{\mathbf x\in\mathcal C_2}\|\W \mathbf x\|_2\le \eta ({r_n^{(2)}}/
 \sqrt n)^2, \mathcal E_K)\le\frac1{n\delta_2}\sum_{k=1}^n
\Pb (\text{\rm dist}(\mathbf X_k,\mathcal H_k)<\eta r_n^{(2)}/\sqrt n, \mathcal E_K ).
\end{align*}
\end{lemma}
\begin{proof}{of Lemma \ref{invertdist}.}
 The proof of this lemma is given in \cite{GotTikh2010}. Note that this proof doesn't use the
independence of the entries of the matrix $\mathbf X$.
For the completeness we repeat this proof here. Recall that $\mathbf X_1,\ldots,\mathbf X_n$
denote the column vector of the matrix $\W$.
Writing $\W\mathbf x=\sum_{k=1}^nx_x\mathbf X_k$, we have
\begin{equation}\label{4101}
 \|\W\mathbf x\|_2\ge \max_{k=1,\ldots,n}\, \text{\rm dist}(x_k\mathbf X_k,\mathcal H_k)
 =\max_{k=1,\ldots,n}\,|x_k|
\text{\rm dist}(\mathbf X_k,\mathcal H_k)
\end{equation}
 Put
\begin{equation}\notag
 p_k=\Pb (\text{\rm dist}(\mathbf X_k,\mathcal H_k)<\eta r_n^{(2)}/\sqrt n ).
\end{equation}
Then
\begin{equation}\notag
 \E|\{k:\text{\rm dist}(\mathbf X_k,\mathcal H_k)<\eta r_n^{(2)}\}|=\sum_{k=1}^np_k.
\end{equation}
Denote by $\mathcal E$ the event that the set $\sigma_1:=\{k:\text{\rm dist}
(\mathbf X_k,\mathcal H_k)\ge \eta r_n^{(2)}/\sqrt n\}$ contains
more than $(1-\delta_2)n$ elements. Then by Chebyshev inequality
\begin{equation}\notag
 \Pb (\mathcal E^{c} )\le \frac1{n\delta_2}\sum_{k=1}^np_k
\end{equation}
 On the other hand, for every incompressible vector $\mathbf x$, the set
 $\sigma_2(\mathbf x)=\{k:\,|x_k|\ge  r_n^{(2)}/\sqrt n\}$ contains at least
$\delta_2n$ elements. (Otherwise, since $\|P_{\sigma_2(\mathbf x)^{c}}\mathbf x\|
\le  r_n^{(2)}$ we have $\|\mathbf x-\mathbf y\|\le
 r_n^{(2)}$ for the sparse vector $\mathbf y=P_{\sigma_2(\mathbf x)^{c}}\mathbf x$).

Assume that the event $\mathcal E$ occurs. Fix any $(\delta_2,r_n^{(2)})$-incompressible
vector $\mathbf x$. Then
$|\sigma_1|+|\sigma_2(\mathbf x)|>(1-\delta_2)n+\delta_2n>n$,  so the sets $\sigma_1$ and
$\sigma_2(\mathbf x)$
have non-empty intersection. Let $k\in\sigma_1\cap\sigma_2(\mathbf x)$.
Then by (\ref{4101} and definitions of
the sets $\sigma_1$ and $\sigma_2(\mathbf x)$, we have
\begin{equation}\notag
 \|\W \mathbf x\|_2\ge|x_k|\text{\rm dist}(\mathbf X_k,\mathcal H_k)\ge
 \eta (r_n^{(2)}/\sqrt n)^2
\end{equation}
 Summarizing, we have shown that
\begin{equation}\notag
\Pb (\inf_{\mathbf x\in\mathcal C_3}\|\W \mathbf x\|_2\le
\eta ({r_n^{(2)}}/\sqrt n)^2 )\le \Pb (\mathcal E^c )\le
\frac1{n\delta_2}\sum_{k=1}^np_k.
\end{equation}
This completes the proof of Lemma \ref{invertdist}.
\end{proof}

To conclude  the proof of Theorem \ref{th:main} it remains to bound the quantity
\begin{equation}\notag
 \gamma_k:=\Pb (\text{\rm dist}(\mathbf X_k, \mathcal H_k)<\eta r_n^{(2)}/\sqrt n, \mathcal E_K ).
\end{equation}
We shall use some modification of Vershynin's approach. We reformulate  the statement of Proposition 5.1 in \cite{Veshyn2011} for  matrices
with correlated entries here.
\begin{statement}\label{Versh}
 Let $\mathbf A=(a_{jk})$ be an arbitrary $n\times n$ matrix. Let $\mathbf A_1$ denote the
 first column of $\mathbf A$ and $\mathcal H_1$
denote the span of the other columns. Furthermore, let $\mathbf B$ denote the
$(n-1)\times(n-1)$ minor of $\mathbf A$ obtained by
removing the first
column and the first row from $\mathbf A$, and let $\mathbf u\in \mathbb R^{n-1}$ and
$\mathbf v^T\in\mathbb R^{(n-1)}$
denote the first column and the first row of $\mathbf A$ respectively with first entry removed.
Then
\begin{equation}
\dist(\mathbf A_1,\mathcal H_1) \geq \frac{|(\B^{-T} \mathbf v , \mathbf u) - a_{11}|}{\sqrt{1+||\B^{-T} \mathbf v||_2^2}}
\end{equation}
\end{statement}
\begin{proof}[Proof of Statement~\ref{Versh}].
The proof of this claim is given in \cite{Veshyn2011}. We repeat this proof here.
Represent the matrix $\A$ in the form
\begin{equation}\label{eq:matrix}
\begin{pmatrix}
a_{11}  & \mathbf{v}^T\\
\mathbf{u}   &  \B \\
\end{pmatrix}
\end{equation}
Let $\mathbf h$ be any unit vector orthogonal to $A_2, ... , A_{n}$. It follows that
$$
\mathbf 0 = \begin{pmatrix}
  \mathbf v^T \\
  \B
\end{pmatrix}^T \mathbf h = h_1 \mathbf v + \B^T \mathbf g,
$$
where $\mathbf h = (h_1, \mathbf g)$, and
$$
\mathbf g = - h_1 \B^{-T} \mathbf v.
$$
From the definition of $\mathbf h$
$$
1 = ||\mathbf h||_2^2 = |h_1|^2 + ||\mathbf g||_2^2 = |h_1|^2  + |h_1|^2 ||\B^{-T} \mathbf v||_2^2
$$
Using this equations we get
$$
\dist(A_1, H) \geq |(A_1,\mathbf h)| = \frac{|a_{11} - (\B^{-T} \mathbf v , \mathbf u)|}{\sqrt{1+||\B^{-T} \mathbf v||_2^2}}
$$
Thus Proposition \ref{Versh} is proved.
\end{proof}
In what follows we set $\A: = \W$. Next, we prove the following
\begin{lemma}\label{key1}
Let the matrix $\W$ denote a matrix  as described in Theorem \ref{th:main}. Let $\mathbf u$, $\mathbf v$
and $\mathbf B$ be determined by~\eqref{eq:matrix}. Then
\begin{equation}\notag
 \sup_{a\in\mathbb R}\Pr\left\{\frac{|({\mathbf B^{-T}} \mathbf v,\mathbf u)-a|}
 {\sqrt{1+\|{\mathbf B^{-T}}\mathbf v\|_2^2}}\le \varepsilon,
\text{ and } \|\mathbf B\|_2\le K_n\right\}\le Cn^{-A}
\end{equation}
with $0<\varepsilon\le n^{-B}$ for some constants $A>0$ and $B>0$.
\end{lemma}

To get this bound we need several statements. First we introduce the matrix and vector
\begin{equation}
\Q = \begin{pmatrix}
\OO_{n-1}  & \B^{-T}\\
\B^{-1} &  \OO_{n-1}\\
\end{pmatrix}
\quad
\mathbf w = \begin{pmatrix}
\mathbf u  \\
\mathbf v
\end{pmatrix},
\end{equation}
where $\OO_{n-1}$ is $(n - 1) \times (n - 1)$ matrix with zero entries. Using the definition of $\Q$ we may write
$$
(\B^{-T} \mathbf v, \mathbf u) = \frac{1}{2} (\Q \mathbf w, \mathbf w).
$$
Introduce vectors
\begin{equation}
\mathbf w' =\begin{pmatrix}
   \mathbf u'\\
   \mathbf v'
\end{pmatrix}
\quad
\mathbf z =\begin{pmatrix}
   \mathbf u\\
   \mathbf v'
\end{pmatrix},
\end{equation}
where $\mathbf u', \mathbf v'$ are independent copies of $\mathbf u, \mathbf v$ respectively. We need the following statement.

\begin{statement}\label{l:q.form decoupling}
$$
\sup_{v \in \R} \Pb_{\mathbf w} \left (|(\Q \mathbf w, \mathbf w) - v| \le 2 \varepsilon \right ) \le \Pb_{\mathbf w, \mathbf w'} \left (|(\Q \OP_{J^{c}}(\mathbf w- \mathbf w'), \OP_J \mathbf w) - u| \le 2 \varepsilon \right ),
$$
where $u$ doesn't depend on $\OP_J \mathbf w = (\OP_J \mathbf u, \OP_J \mathbf v)^T$.
\end{statement}

\begin{remark}
 This Lemma was stated and proved in \cite[Statement~2]{Nau2013b}. We repeat here the proof.
\end{remark}

\begin{proof}
Let us fix $v$ and denote
$$
p:=\Pb \left (|(\Q\mathbf w, \mathbf w) - v| \le  2 \varepsilon \right ).
$$
We can decompose the set $[n]$ into the union $[n] = J \cup J^c$. We can take $\mathbf u_1 = \OP_J \mathbf u, \mathbf u_2 = \OP_{J^c} \mathbf u, \mathbf v_1 = \OP_J \mathbf v$ and $\mathbf v_2 = \OP_{J^c} \mathbf v$.
In what follows we shall use a simple ``decoupling'' inequality. Let $X,Y$ be independent r.v.'s and denote by $X'$ an independent copy of $X$, which is also independent from $Y$.
Let $\mathcal E(X,Y)$ be an event depending on $X$ and $Y$. Then
\begin{equation}\label{sym}
{\Pb}^2\{\mathcal E(X,Y)\}\le {\Pb}\{\mathcal E(X,Y)\cap\mathcal E(X',Y)\}.
\end{equation}
This inequality was stated in \cite{Veshyn2011}, Lemma~8.5, and \cite[Lemma~14]{Costello2013}.
It originated in~\cite[Lemma~3.37]{Gotze1979}. The proof is very simple. The claim follows from the inequality
\begin{align*}
\Pb^2 \{\mathcal E(X,Y)\}&=(\E\Pb\{\mathcal E(X,Y)|Y\})^2 \le \E {\Pb}^2\{\mathcal E(X,Y)|Y\}\\&=
\E\Pb\{\mathcal E(X,Y)\cap\mathcal E(X',Y)|Y\}= {\Pb}\{\mathcal E(X,Y)\cap\mathcal E(X',Y)\}.
\end{align*}
Applying inequality \eqref{sym}, we get
\begin{align} \label{eq:p}
&p^2 \le \Pb \left (|(\Q \mathbf w, \mathbf w) - v| \le  2 \varepsilon , |(\Q   \mathbf z,  \mathbf z) - v| \le  2 \varepsilon \right )\\
&\le \Pb \left (|(\Q \mathbf w, \mathbf w)  - (\Q \mathbf z, \mathbf z)| \le  4 \varepsilon \right ) \nonumber .
\end{align}
Let us rewrite $\B^{-T}$ in the block form
$$
\B^{-T} = \begin{pmatrix}
\EE  & \F\\
\G &  \HH\\
\end{pmatrix}.
$$
We get
\begin{align*}
(\Q \mathbf w, \mathbf w) &= (\EE \mathbf v_1 , \mathbf u_1) + (\F \mathbf v_2, \mathbf u_1) + (\G \mathbf v_1, \mathbf u_2) + (\HH \mathbf v_2, \mathbf u_2)\\& +
 (\EE^T \mathbf u_1, \mathbf v_1) + (\G^T \mathbf u_2, \mathbf v_1) + (\F^T \mathbf u_1, \mathbf v_2) + (\HH^T \mathbf u_2, \mathbf v_2)\\
(\Q \mathbf z, \mathbf z) &= (\EE \mathbf v_1, \mathbf u_1) + (\F \mathbf v_2', \mathbf u_1) + (\G \mathbf v_1, \mathbf u_2') + (\HH \mathbf v_2', \mathbf u_2')  \\& +
(\EE^T \mathbf u_1, \mathbf v_1) + (\G^T \mathbf u_2', \mathbf v_1) + (\F^T \mathbf u_1, \mathbf v_2') + (\HH^T \mathbf u_2', \mathbf v_2')
\end{align*}
and
\begin{align}\label{eq:q.f.1}
(\Q \mathbf w, \mathbf w) - (\Q \mathbf z,\mathbf z) &= 2(\F (\mathbf v_2 - \mathbf v_2'), \mathbf u_1) + 2 (\G^T (\mathbf u_2 - \mathbf u_2'), \mathbf v_1)\\& +
 2 (\HH \mathbf v_2, \mathbf u_2) - 2(\HH \mathbf v_2', \mathbf u_2') \nonumber .
\end{align}
The last two terms in~\eqref{eq:q.f.1} depend only on $\mathbf u_2, \mathbf u_2',\mathbf v_2,\mathbf v_2'$ and we conclude that
$$
p_1^2 \le \Pb \left ( |(\Q P_{J^c}(\mathbf w - \mathbf w'), P_J \mathbf w ) - u| \le 2 \varepsilon \right),
$$
where $u = u(\mathbf u_2,\mathbf v_2,\mathbf u_2',\mathbf v_2',\HH)$.
\end{proof}

\begin{statement} \label{l:incomp of inverse}
For all $\mathbf u \in \R^{n}$ there exists a constant $c>0$ such that
$$
\Pb\left( \frac{\B^{-T} \mathbf u}{||\B^{-T}\mathbf u||_2} \notin \mathcal C_2 \text{ and } ||\B|| \le K_n \right) \le e^{-c n }.
$$
\end{statement}
\begin{proof}
Note that
\begin{equation}\notag
 \mathcal S^{(n-1)}=\mathcal C_0\cup\mathcal C_1\cup\mathcal C_2.
\end{equation}
Let $ \mathbf x = \B^{-T} \mathbf u$. It is easy to see that
\begin{align*}
\left \{\frac{\B^{-T} \mathbf u}{||\B^{-T}  \mathbf u||_2} \notin \mathcal C_2 \right \} \subseteq \left \{\exists  \mathbf x: \frac{ \mathbf x}{|| \mathbf x||_2} \in \mathcal C_0 \cup \mathcal C_1 \text{ and } \B^T  \mathbf x =  \mathbf u \right \}
\end{align*}
This implies, for $c>0$,
\begin{align*}
&\Pb \left( \frac{\B^{-T} \mathbf u}{||\B^{-T}\mathbf u||_2} \notin \mathcal C_2 \text{ and } ||\B|| \le K_n \right) \\
&\le \sum_{i=0}^1 \Pb\left( \inf_{\frac{\mathbf x}{||\mathbf x||_2} \in \mathcal C_i} || \B^T \mathbf x - \mathbf u||_2 / ||\mathbf x||_2 \le \tau \text{ and } ||\B|| \le K_n \right )
\end{align*}
Now choose $\tau$ small enough to apply Lemmas~\ref{help1} and~\ref{help2}.
\end{proof}

\begin{remark}
It is not very difficult to show that for all $\mathbf w \in \R^{2n}$
$$
\Pb\left( \frac{\Q \mathbf w}{||\Q \mathbf w||_2} \notin \mathcal C_2 \text{ and } ||\B|| \le K_n \right) \le e^{-c n }.
$$
To prove that one should observe that
$$
\Q^{-1} = \begin{pmatrix}
\OO_{n-1}  & \B\\
\B^{T} &  \OO_{n-1}\\
\end{pmatrix}
$$
and repeat the proof of Lemmas~\ref{local}--\ref{incomp1} for the matrix $\W$ replaced by $\Q^{-1}$.
\end{remark}
\begin{statement} \label{l:norm est}
Let $\W$ satisfy condition $\Cond$ and consider a matrix $\B$  and a vector $\mathbf v$  determined by the decomposition~\eqref{eq:matrix}.
Assume that $||\B|| \le  K_n$. Let $\mathbf x=(X_1,\ldots, X_n)$, where
$X_1\ldots, X_n$ are independent r.v.'s with $\E X_j=0$ and $\E X_j^2=1$. We shall assume
that the r.v.'s $X_j$ are independent from the matrix $\B$. Then with probability at least $1 - e^{-c_0 n}$ the matrix $\B$ has the following properties:
\begin{itemize}
\item[a)] $\Pb_{\mathbf v}(\sqrt{1+||\B^{-T} \mathbf v||_2^2} \leq \varepsilon^{-1/2} ||\B^{-T}||_2) \geq 1 - n^{-1/4}$,\\
\item[b)] $\Pb_{\mathbf x}(||\B^{-T} \mathbf x||_2 \geq \varepsilon ||\B^{-1}||_{2}) \geq 1 - c n^{-1/4}$,
\end{itemize}
where $\varepsilon < \min \left (\frac{n^{-1/4}}{2 K_n^2}, c'\frac{r_n^{(2)}}{n^{1/4}}\right )$.
\end{statement}

\begin{proof}
Let $\{\mathbf e_k\}_{k=1}^n$ be a standard basis in $\R^n$. For all $1 \le k \le n$ define vectors by
$$
\mathbf a_k := \frac{\B^{-1} \mathbf e_k}{||\B^{-1} \mathbf e_k||}.
$$
By Statement~\ref{l:incomp of inverse} the vector $\mathbf a_k$ is incompressible with probability $1 - e^{-c_0 n}$. Fix a  matrix $\B$ with this property.\\
a) By Chebyshev inequality
$$
\Pb_{\mathbf v}(\sqrt{1+||\B^{-T} \mathbf v||_2^2} \geq \varepsilon^{-1/2} ||\B^{-T}||_2) \le 2\varepsilon K_n^2.
$$
We may choose $\varepsilon < \frac{n^{-1/4}}{2 K_n^2}$.\\
b) Note that
$$
||\B^{-T}\mathbf x ||_2^2
=\sum_{k=1}^n\|\B^{-1}\mathbf e_k\|_2^2 (\mathbf a_k,\mathbf x)^2.
$$
We may conclude from Lemma~\ref{ap:l_conc_funct_2} that
$$
\Pb(|(\mathbf a_k, \mathbf x)| \le \varepsilon) \le c n^{-1/4}.
$$
for all $\varepsilon < c' n^{-1/4} r_n^{(2)}$. Applying now~\cite[Lemma 8.3]{Veshyn2011} with
$p_k=\frac{\|\B^{-1}\mathbf e_k\|_2^2}{\|\B^{-1}\|_2^2}$, we get
\begin{align*}
&\Pb_{\mathbf x}(||\B^{-T} \mathbf x||_2 \le \varepsilon ||\B^{-1}||_{2} ) = \Pb_{\mathbf x}(||\B^{-T} \mathbf x||_2^2 \le \varepsilon^2 ||\B^{-1}||_{2}^2 ) \\
&= \Pb_{\mathbf x}(\sum_{k=1}^n ||\B^{-1} \mathbf e_k||_2^2 (\mathbf a_k, \mathbf x)^2 \le \varepsilon^2 ||\B^{-1}||_{2}^2 ) =\Pb_{\mathbf x}(\sum_{k=1}^n p_k (\mathbf a_k, \mathbf x)^2 \le \varepsilon^2  )\\
&\le 2\sum_{k=1}^n p_k \Pb_{\mathbf x}( |(\mathbf a_k, \mathbf x)| \le \sqrt 2 \varepsilon) \le c n^{-1/4}.
\end{align*}
\end{proof}

\begin{proof}[Proof of Lemma~\ref{key1} ]
Let $\xi_1, ... ,\xi_n$ be i.i.d. Bernoulli random variables  with $\E \xi_i = c_{0}/2$, where $c_{0}$ is some constant which will be chosen later.
Define the set $\mathcal J: =\{i: \xi_i = 0 \}$ and the event $\mathcal E_0 : = \{|\mathcal J^c| \le c_0 n \}$. From a large deviation inequality we may conclude that
$\Pb(\mathcal E_0) \geq 1 - 2 \exp(-c_0^2 n/2)$. Introduce the event
$$
\mathcal E_1:= \{\varepsilon_0^{1/2} \sqrt {1 + ||\B^{-T} \mathbf v||_2^2} \le ||\B^{-1}||_{2} \le \varepsilon_0^{-1}||\Q \OP_{\mathcal J^c} (\mathbf w - \mathbf w')||_2 \},
$$
where $\varepsilon_0$ will be chosen later.

From Statement~\ref{l:norm est} we may conclude that
$$
\Pb_{\B, \mathbf w, \mathbf w', \mathcal J} (\mathcal E_1 \cup ||\B|| \geq K_n) \geq 1 - C' n^{-1/4} - 2 e^{-c' n}.
$$
Consider the random vector
$$
\mathbf w_0 = \frac{1}{||\Q \OP_{\mathcal J^c} (\mathbf w - \mathbf w')||_2}\begin{pmatrix}
  \B^{-T} \OP_{\mathcal J^c} (\mathbf v - \mathbf v') \\
  \B^{-1} \OP_{\mathcal J^c} (\mathbf u - \mathbf u')
\end{pmatrix} =
\begin{pmatrix}
   \mathbf a\\
   \mathbf b
\end{pmatrix}.
$$
By the remark after Statement~\ref{l:incomp of inverse} it follows that the event $\mathcal E_2:=\{\mathbf w_0 \in \incomp(\delta,r_n^{(2)}) \}$ holds with probability
$$
\Pb_{\B}(\mathcal E_2 \cup ||\B|| \geq K_n | \mathbf w, \mathbf w', \mathcal J) \geq 1 -  2 \exp(-c'' n).
$$
Combining these probabilities we have
\begin{align*}
& \Pb_{\B, \mathbf w, \mathbf w', \mathcal J} (\mathcal E_0, \mathcal E_1, \mathcal E_2 \cup ||\B|| \geq K_n) \\
& \geq 1 -  2 e^{-c_{0}^2 n/2 } - C'n^{-1/4} - 2 e^{-c' n} - 2 e^{-c'' n}: = 1 - p_0.
\end{align*}
We may fix $\mathcal J$ that satisfies $|\mathcal J^c| \le c_0 n$  and
$$
\Pb_{\B,\mathbf w, \mathbf w'} (\mathcal E_1, \mathcal E_2 \cup ||\B|| \geq K_n ) \geq 1 - p_0.
$$
By Fubini's theorem $\B$ has the following property with probability at least $1 - \sqrt{p_0}$
$$
\Pb_{\mathbf w, \mathbf w'} (\mathcal E_1, \mathcal E_2 \cup ||\B|| \geq K_n | \B ) \geq 1 - \sqrt{p_0}.
$$
The event $\{||\B|| \geq K_n\}$ depends only on $\B$. We may conclude that the random matrix $\B$ has the following property with probability at least $1 - \sqrt{p_0}$: either $||\B|| \geq K_n$, or
\begin{align} \label{eq:B prop}
||\B|| \le K_n \text{ and } \Pb_{\mathbf w, \mathbf w'} (\mathcal E_1, \mathcal E_2| \B) \geq 1 - \sqrt p_0
\end{align}
The event we are interested in is
$$
\Omega_0 := \left ( \frac{|(\Q \mathbf w, \mathbf w) - u|}{\sqrt{1 + ||\B^{-T} \mathbf v||_2^2}} \le 2 \varepsilon \right ).
$$
We need to estimate the probability
$$
\Pb_{\B, \mathbf w} (\Omega_0 \cap ||\B|| \le K_n) \le \Pb_{\B, \mathbf w} (\Omega_0 \cap \text{ \eqref{eq:B prop} holds} ) + \Pb_{\B, \mathbf w} (||\B|| \geq K_n \cap \text{ \eqref{eq:B prop} fails} ).
$$
The last term is bounded by $\sqrt{p_0}$.
$$
\Pb_{\B, \mathbf w} (\Omega_0 \cap ||\B|| \le K_n) \le \sup_{\B \text{ satisfies~\eqref{eq:B prop}}} \Pb_{\mathbf w} (\Omega_0| \B ) + \sqrt{p_0}.
$$
We conclude that
\begin{align*}
\Pb_{\B, \mathbf w} (\Omega_0 \cap ||\B|| \le K_n) \le \sup_{\B \text{ satisfies~\eqref{eq:B prop}}} \Pb_{\mathbf w, \mathbf w'} (\Omega_0, \mathcal E_1| \B ) + 2\sqrt{p_0}.
\end{align*}
Let us fix $\B$ that satisfies~\eqref{eq:B prop} and denote $p_1: = \Pb_{\mathbf w, \mathbf w'} (\Omega_0, \mathcal E_1| \B )$. By Statement~\ref{l:q.form decoupling} and
the first inequality in $\mathcal E_1$ we have
\begin{align*}
p_1^2 \le \Pb_{\mathbf w, \mathbf w'} \left ( \underbrace{|(\Q \OP_{\mathcal J^c} (\mathbf w - \mathbf w'), \OP_{\mathcal J} \mathbf w) - v| \le \frac{2\varepsilon}{\sqrt{\varepsilon_0} } ||\B^{-1}||_{2}}_{\Omega_1} \right )
\end{align*}
and
\begin{align*}
\Pb_{\mathbf w, \mathbf w'} ( \Omega_1 ) \le \Pb_{\mathbf w, \mathbf w'} ( \Omega_1, \mathcal E_1, \mathcal E_2) + \sqrt{p_0}.
\end{align*}
Furthermore,
$$
p_1^2 \le \Pb_{\mathbf w, \mathbf w'} ( |(\mathbf w_0, \OP_{\mathcal J} \mathbf w) - v| \le 2\varepsilon_0^{-3/2} \varepsilon, \mathcal E_2 ) + \sqrt{p_0}.
$$
By definition the random vector $\mathbf w_0$ is determined by the random vector $\OP_{\mathcal J^c}(\mathbf w - \mathbf w')$, which is independent of
the random vector $\OP_{\mathcal J} \mathbf w$. Fix the vector $\OP_{\mathcal J^c}(\mathbf w - \mathbf w')$, obtaining
$$
p_1^2 \le \sup_{\substack {\mathbf w_0 \in \incomp(\delta,r_n^{(2)})\\ w \in \R}} \Pb_{\OP_{\mathcal J} \mathbf w} \left ( |(\mathbf w_0, \OP_{\mathcal J} \mathbf w) - w| \le 2\varepsilon_0^{-3/2} \varepsilon \right ) + \sqrt{p_0}.
$$
Let us fix a vector $\mathbf w_0$ and a number $w$. We may rewrite
\begin{equation}\label{eq:sum}
(\mathbf w_0, P_\mathcal{J} \mathbf w)  = \sum_{i \in \mathcal J} ( a_i X_i + b_i Y_i),
\end{equation}
where $||\mathbf a||_2^2 + ||\mathbf b||_2^2 = 1$. For an arbitrary set $I$ we introduce the notation $S_I := \sum_{i \in I} (a_i X_i + b_i Y_i)$. Denote the concentration function of $S_I$ by
\begin{equation}
 Q(S_I,\lambda)=\sup_{a\in\mathbb R}\Pb (|S_I-a|\le \lambda ).
\end{equation}
Our aim is to estimate $Q(S_\mathcal{J},\widetilde\varepsilon)$. For this we would like to apply Lemma~~\ref{ap:l.concentration}, but we can't do it directly. We first need to get appropriate estimates
for the coefficients $a_i, b_i, i \in \mathcal J$.

From Lemma~\ref{l,a:incomp vec} we know that at least $n\delta$ coordinates of vector $w_0 \in \incomp(\delta,r_n^{(2)})$ satisfy
$$
\frac{r_n^{(2)}}{2\sqrt{n}} \le |w_{0i}| \le \frac{1}{\sqrt{\delta n}}.
$$
Denote the spread set of the vectors $\mathbf a$ and $\mathbf b$ by $\sigma(\mathbf a)$ and $\sigma(\mathbf b)$ respectively. It is easy to see that $\max(\sigma(\mathbf a), \sigma(\mathbf b)) \geq \delta n / 2$. Without loss of generality assume that $\sigma(\mathbf a) \geq \delta n /2$. Let us denote by $\mathcal J_0: = \mathcal J \cap \sigma(\mathbf a)$. We may take $c_0 = \delta /4$ and conclude by construction of $\mathcal J$ that $|\mathcal J_0| \geq [\delta n/4]$.
By Lemma~\ref{l,a:reduction} we get $Q(S_\mathcal{J},\widetilde\varepsilon) \le Q(S_{\mathcal{J}_0},\widetilde\varepsilon)$ . The only information we know about the vector $b$ is that
$$
0 \le |b_i| \le 1 \text{ for all } i \in \mathcal J_0.
$$
But from the inequality $||b||_2^2 \le 1$ we may conclude that we can restrict the sum $S_{\mathcal{J}_0}$ on the set
$$
\widetilde{\mathcal{J}}_0:= \left \{i \in \mathcal J_0: |b_i| \le \frac{c}{n^{1/4}} \right \}
$$
and $|\widetilde{\mathcal{J}}_0| \geq c n$. Now we split the set $\widetilde{\mathcal{J}}_0$ as a union of the disjoint sets $J_l$:
\begin{align*}
&J_l = \left \{i \in \widetilde{\mathcal{J}}_0: \frac{2^{l-1} r_n^{(2)}}{2\sqrt{n}} \le |a_i| \le \frac{2^l r_n^{(2)}}{2\sqrt{n}} \right \}, l = 1, ..., L,\\
&J_{L+1} = \left \{i \in \widetilde{\mathcal{J}}_0: \frac{2^{L} r_n^{(2)}}{2\sqrt{n}} \le |a_i| \le \frac{1}{\sqrt{\delta n}} \right \},
\end{align*}
where $L = [c(Q) \ln n]$. It follows from the Dirichlet principle  that there exists $l_0, 1 \le l_0 \le L+1$, such that $|J_{l_o}| \geq c n\ln^{-1} n$. Let's set $\mathcal{I}:=J_{l_0}$.
Again by Lemma~\ref{l,a:reduction} we may write  $Q(S_{\mathcal{J}_0}, \widetilde\varepsilon) \le Q(S_{\mathcal{I}}, \widetilde\varepsilon)$.
Set
$$
\eta_n = \frac{2\sqrt{n}}{2^{l_0-1} r_n^{(2)}}.
$$
It is easy to see that $\eta_n \geq c \sqrt n$.

Introduce the notations $\widetilde a_i:= \eta_n a_i$ and $\widetilde b_i:= \eta_n b_i$. It follows that
$$
1 \le |\widetilde a_i| \le 2, |\widetilde b_i| \le \eta_n n^{-1/4}.
$$
Let $\widetilde S_I = \sum_{i \in I} (\widetilde a_i X_i + \widetilde b_i Y_i)$, where $I$ is an arbitrary set.
We decompose the set $\mathcal{I}$ into the sum of two sets $\mathcal{I} = \mathcal{I}_{1} \cup \mathcal{I}_2$, where
$$
\mathcal{I}_1: = \{i \in \mathcal{I}: |\widetilde b_i| \le 2\} \text{ and } \mathcal{I}_2: = \{i \in \mathcal{I}: 2<|\widetilde b_i| \le \eta_n n^{-1/4} \}.
$$
We have $\max(|\mathcal{I}_1|, |\mathcal{I}_2|) \geq c n \ln^{-1} n /2$.
From the properties of  concentration functions it follows
$$
Q(\widetilde S_{\mathcal{I}}, \varepsilon \eta_n) \le \min(Q(\widetilde S_{\mathcal{I}_1}, \widetilde \varepsilon \eta_n), Q(\widetilde S_{\mathcal{I}_2}, \widetilde \varepsilon \eta_n)).
$$
To finish the proof of the statement we shall consider two cases.

1) Suppose that  $|\mathcal{I}_1| \geq c n \ln^{-1} n /2$. Denote by $\sigma^2 = \E \widetilde S_{\mathcal{I}_1}^2$.
Elementary calculations show that on the set $\mathcal{I}_1$ we have $\sigma^2 \geq (1-\rho^2)|\mathcal{I}_1|$. We may apply Lemma~\ref{ap:l.concentration}
$$
Q(\widetilde S_{\mathcal{I}_1},\widetilde \varepsilon \eta_n) \le \frac{\sqrt {\widetilde \varepsilon \eta_n}}{\sqrt{|\mathcal{I}_1|}(2(1 - \rho^2) - 8 \max_k \E Z_k^2 \mathbb I{\{|Z_k| \geq \widetilde \varepsilon \eta_n/2\}})^{1/2}},
$$
where $Z_i = \widetilde a_i X_i + \widetilde b_i Y_i$.
We can estimate the maximum in the denominator in the following way
\begin{equation} \label{eq: max in denominator_1}
\max_k \E Z_k^2 \mathbb I{\{|Z_k| \geq \widetilde \varepsilon \eta_n/2\}} \le c \max_k \E X_k^2 \mathbb I{\{|X_k| \geq \widetilde\varepsilon \eta_n/4\}}.
\end{equation}
We take
$$
\widetilde \varepsilon = \varepsilon_1:=\frac{M_1 n^{1/4}}{\eta_n \ln n},
$$
where $M_1$ is some constant. 
Note that $\widetilde \varepsilon \le c' n^{-1/4} \ln^{-1} n$ and $\widetilde \varepsilon \eta_n = M_1 n^{1/4} \ln^{-1} n$. Choosing $M_1$, the right hand side of~\eqref{eq: max in denominator_1} 
can be made as small as $1 - \rho^2$, assuming that $X_k$ has uniformly integrable second moment. We conclude that
$$
Q(\widetilde S_{\mathcal{I}_1},\widetilde \varepsilon \eta_n)\le n^{-A_1},
$$
for some $A_1 > 0$.

2) Suppose that $|\mathcal{I}_2| \geq c n \ln^{-1} n/2$. For the set $\mathcal{I}_2$ we repeat the above procedure and find the set $I_{k_0}$ such that
$$
2^{k_0} \le |b_i| \le 2^{k_0+1} \text{ for all } i \in I_{k_0}
$$
and $|I_{k_0}| \geq c n\ln^{-2} n$.
We may reduce the sum $\widetilde S_{\mathcal{I}_2}$ to this set and  $Q(\widetilde S_{\mathcal{I}_2},\widetilde \varepsilon \eta_n)  \le Q(\widetilde S_{I_{k_0}}, \widetilde \varepsilon \eta_n)$.
On the set $I_{k_0}$ the variance $\sigma^2 = \E \widetilde S_{\mathcal{I}_2}^2$ is bounded below by
$$
\sigma^2 \geq (1 - \rho^2) ||b||_2^2 \geq (1 - \rho^2) |I_{k_0}| 2^{2k_0}
$$
Applying Lemma~\ref{ap:l.concentration} we get
$$
Q(\widetilde S_{I_{k_0}},\widetilde \varepsilon \eta_n)\le \frac{\sqrt{\widetilde \varepsilon \eta_n }}{\sqrt{|I_{k_0}|}(2(1 - \rho^2) 2^{2k_0} - 8 \max_k \E Z_k^2 \mathbb I{\{|Z_k| \geq \widetilde \varepsilon \eta_n/2\}})^{1/2}}.
$$
Note that
\begin{equation} \label{eq: max in denominator_2}
\max_k \E Z_k^2 \mathbb I{\{|Z_k| \geq \widetilde \varepsilon \eta_n\}} \le c 2^{2k_0} \max_k \E X_k^2 \mathbb I{\{|X_k| \geq \widetilde \varepsilon \eta_n 2^{-k_0-1}/4\}}
\end{equation}
We take
$$
\widetilde \varepsilon = \varepsilon_2:= \frac{M_2 n^{1/8} 2^{k_0}}{\eta_n \ln^2 n},
$$
where $M_2$ is some constant. One may check that $\widetilde \varepsilon \le c'' n^{-1/8}\ln^{-2}n$ and $\widetilde \varepsilon \eta_n 2^{-k_0} = M_2 n^{1/8} \ln^{-2} n$. Choosing $M_2$, the right hand side of~\eqref{eq: max in denominator_2} can be made as small as $(1 - \rho^2)2^{2k_0}$, assuming that $X_k$ has a uniformly integrable second moment. We conclude that
$$
Q(\widetilde S_{\mathcal{I}_1},\widetilde \varepsilon \eta_n)\le n^{-A_2},
$$
for some $A_2 > 0$.

Now we take $\widetilde \varepsilon = \min(\varepsilon_1, \varepsilon_2)$ and conclude the statement.

\end{proof}

\section{Application to the elliptic law}

In this section we briefly discuss the application of Theorem~\ref{th:main} to the elliptic law.

Denote by $\lambda_1, ..., \lambda_n$ the eigenvalues of the matrix $n^{-1/2} \X$ and define the empirical spectral measure of eigenvalues by
$$
\mu_n(B) = \frac{1}{n} \# \{1 \le i \le n: \lambda_i \in B \}, \quad B \in \mathcal{B}(\C),
$$
where $\mathcal{B}(\C)$ is a Borel $\sigma$-algebra of $\C$.

We say that a sequence of random probability measures $m_n(\cdot)$
converges weakly in probability to the probability measure $m(\cdot)$ if for all continuous and bounded functions $f: \C \rightarrow \C$ and all $\varepsilon > 0$
$$
\lim_{n \rightarrow \infty}\Pb \left ( \left | \int_\C f(x) m_{n}(dz) - \int_\C f(x)m(dz) \right | > \varepsilon \right ) = 0.
$$
We denote weak convergence by the symbol $\xrightarrow{weak}$.

A fundamental problem in the theory of random
matrices is to determine the limiting distribution of $\mu_n$
as the size of the random matrix tends to infinity. The following theorem gives the solution of this problem for matrices which satisfy the conditions $\Cond$ and $\UI$.
\begin{theorem} {\bf (Elliptic Law)}\label{th:elliptic main}
Let the entries $X_{jk}, 1\le j,k \le n$, of the matrix $\X$ satisfy the conditions $\Cond$ and $\UI$. Assume that $|\rho| < 1$. Then $\mu_n \xrightarrow{weak} \mu$ in probability, and $\mu$ has the density $g$:
$$
g(x, y) = \begin{cases}
  \frac{1}{\pi (1 - \rho^2)}, & x, y \in \mathcal E, \\
  0, & \text{otherwise,}
\end{cases}
$$
where
$$
\mathcal E := \left \{ u,v \in \R: \frac{u^2}{(1+\rho)^2} + \frac{v^2}{(1-\rho)^2} \le 1 \right \}.
$$
\end{theorem}

\begin{figure}
\begin{center}
\scalebox{.43}{\includegraphics{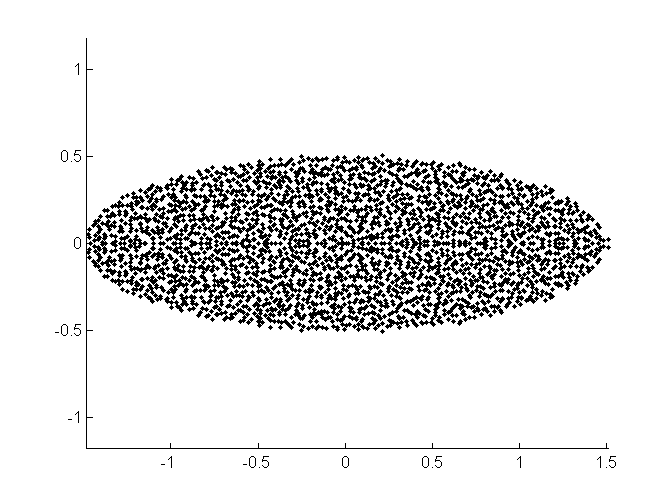}}
\scalebox{.43}{\includegraphics{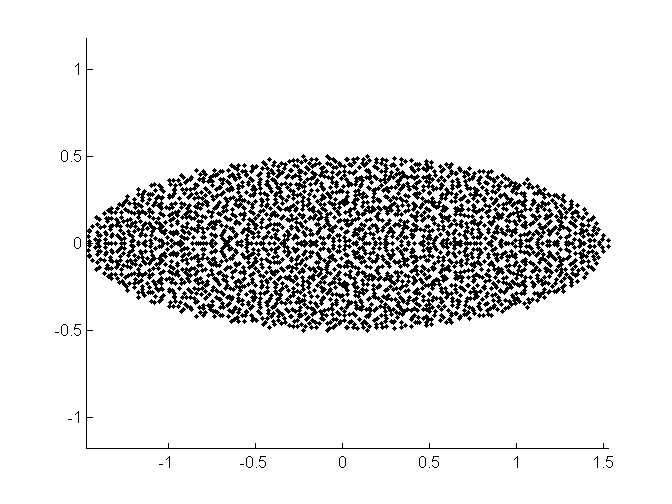}}
\end{center}
\caption{Eigenvalues of the matrix $n^{-1/2} \X$ for $n = 3000$ and $\rho = 0.5$.
On the left, each entry is an iid Gaussian normal random variable.
On the right, each entry is an iid Bernoulli random variable, taking the values $+1$ and $-1$ each with probability $1/2$. }
\label{fig:ellip1}
\end{figure}

Theorem~\ref{th:elliptic main} asserts that under assumptions $\Cond$ and $\UI$ the empirical spectral measure weakly converges in probability to the uniform distribution on the ellipse.
The axis of the ellipse are determined by the correlation $\E X_{jk} X_{k j} = \rho, 1 \le j < k \le n$. This result was called by Girko ``Elliptic Law''.
The limiting distribution doesn't depend on the distribution of the matrix elements and in this sense the result   is universal.

In 1985 Girko proved the elliptic law for rather general ensembles of random matrices under the assumption
that the matrix elements have a density, see~\cite{Girko1985} and~\cite{Girko2006}.
Girko used the method of characteristic functions. Using
a so-called  $V$-transform he reduced the problem
to the problem for Hermitian matrices $(n^{-1/2} \X - z \I)^{*}(n^{-1/2} \X - z \I)$ and
established the convergence of the empirical spectral distribution of singular values of $n^{-1/2}\X - z \I$ to some
limit which determines the elliptic law.  Under the assumption of a finite fourth moment  the elliptic law was recently proved by Naumov in~\cite{Nau2012}. Recently Nguyen and O'Rourke in~\cite{ORourkeNguyen2012} have proved the result of Theorem~\ref{th:elliptic main} for i.i.d. r.v.'s, assuming the second moment finite. Applying the result of Theorem~\ref{th:main} we may extend the elliptic law to the class of random matrices with non i.i.d. entries which satisfy the conditions $\Cond$ and $\UI$.

Below we shall provide a short outline of the proof of Theorem~\ref{th:elliptic main}.
\subsection{Gaussian case} \label{sec: gaussian case}

Assume that the  elements of a real random matrix $\X$ have Gaussian distribution with zero mean and correlations
$$
\E X_{i j}^2 = 1 \text{ and } \E X_{i j} X_{i j} = \rho, \quad 1 \le i < j \le n,  \quad |\rho| < 1.
$$
In~\cite{SomStein1988} it was shown that the ensemble of such matrices can be specified by the probability measure
$$
\Pb(d\A) \thicksim \exp \left [-\frac{n}{2(1 - \rho^2)} \Tr(\A \A^T - \rho \A^2) \right ],
$$
on the space of matrices $\A$. It was proved that $\mu_n \xrightarrow{weak} \mu$, where $\mu$ has a density from Theorem~\ref{th:main}, see~\cite{SomStein1988} for details. We will use this result to prove Theorem~\ref{th:elliptic main} in the general case.

For the discussion of the elliptic law in the Gaussian case see also~\cite{FyodSomm1998},~\cite[Chapter~18]{Akem2011} and~\cite{Ledoux2008}.

\subsection{Proof of the elliptic law}
In the general case 
we have to consider elements of the matrix $\X$ with arbitrary distributions. To overcome this difficulty we shall use the method of logarithmic potentials. Let us denote by $s_1 \geq s_2 \geq ... \geq s_n$ the singular values of $n^{-1/2}\X - z \I$ and introduce the empirical spectral measure $\nu_n(\cdot, z)$ of the singular values.  The convergence in Theorem~\ref{th:elliptic main} will be proved via the convergence of the logarithmic potential of $\mu_n$ to the logarithmic potential of $\mu$.
We can rewrite the logarithmic potential of $\mu_n$ via the logarithmic moments of measure $\nu_{n}$ by
\begin{align*} \label{eq:log. potential}
&U_{\mu_n} (z) = -\int_\C \ln|z - w| \mu_n(dw) = -\frac{1}{n} \ln \left |\det \left (\frac{1}{\sqrt n} \X - z\I \right )\right| \\
&= - \frac{1}{2 n} \ln \det \left (\frac{1}{\sqrt n} \X-z\I \right )^{*}\left (\frac{1}{\sqrt n} \X-z\I \right ) =
-\int_0^\infty \ln x \nu_{n}(dx).
\end{align*}
This allows us to consider the Hermitian matrices $(n^{-1/2} \X-z\I)^{*}(n^{-1/2} \X-z\I)$ instead of $n^{-1/2} \X$. To prove Theorem~\ref{th:main} we need the following lemma.
\begin{lemma} \label{l:Girko}
Suppose that for a.a. $z \in \mathbb C$ there exists a probability measure $\nu_z$ on $[0, \infty)$ such that\\
a) $\nu_n \xrightarrow{weak} \nu_z$ as $n \rightarrow \infty$ in probability\\
b) $\ln$ is uniformly integrable in probability with respect to $\{\nu_n\}_{n \geq 1}$.

Then there exists a probability measure $\mu$ such that\\
a) $\mu_{n} \xrightarrow{weak} \mu$ as $n \rightarrow \infty$ in probability \\
b) for a.a. $z \in \C$
$$
U_{\mu}(z) = - \int_{0}^\infty \ln x \nu_z(dx).
$$
\end{lemma}
\begin{proof}
See \cite{BordCh2011}[Lemma~4.3] for the proof.
\end{proof}

Suppose now that $\X$ satisfies the conditions a) and b) of Lemma~\ref{l:Girko} and the measure $\nu_z$ is the same for all matrices which satisfy the conditions $\Cond$ and $\UI$. Then there exist a probability measure $\hat \mu$ such that $\mu_n \xrightarrow{weak} \hat \mu$ and
$$
U_{\hat \mu}(z) = - \int_{0}^\infty \ln x \nu_z(dx).
$$
We know that in the Gaussian case $\mu_n$ converges to the elliptic law  $\mu$. Due to the assumption that $\nu_z$ is the same for all matrices which satisfy the conditions $\Cond, \UI$ and Lemma~\ref{l:Girko} we have
$$
U_{\mu}(z) = - \int_{0}^\infty \ln x \nu_z(dx)
$$
We get $U_{\hat \mu}(z) = U_{\mu}(z)$. From the uniqueness of the logarithmic potential we may conclude the statement of the Theorem.

It remains to check the assumptions we have made in the beginning of the proof.  The following lemma proves the condition a) of Lemma~\ref{l:Girko} and shows that $\nu_z$ is the same for all matrices which satisfy the conditions $\Cond$ and $\UI$.

We say the entries $X_{j,k}$, $1\le j,k\le n$, of the matrix $\X$ satisfy Lindeberg's condition $\Lind$ if
$$
\text{for all $\tau>0$ } \quad \frac{1}{n^2}\sum_{i,j=1}^n \E X_{ij}^2 \mathbb I(|X_{ij}| \geq \tau \sqrt n) \rightarrow 0 \text{ as } n \rightarrow \infty.
$$
It easy to see that $\UI \Rightarrow \Lind$

Let $\mathcal{F}_n(x,z)$ be the empirical distribution function of the singular values $s_1 \geq ... \geq s_n$
of the matrix $n^{-1/2}\X - z\I$ which corresponds to the measure $\nu_n(z, \cdot)$.
\begin{lemma}\label{th:Stieltjes transform}
Let $X_{jk}$, $1\le j,k\le n$, satisfy the conditions \Cond and \Lind. Then there exists a non-random distribution function $\mathcal{F}(x,z)$ such that
for all continuous and bounded functions $f(x)$, a.a. $z \in \C$ and all $\varepsilon > 0$
\begin{equation*}
\Pb \left ( \left |\int_\R f(x) d\mathcal{F}_n(x,z) - \int_\R f(x) d\mathcal{F}(x,z) \right | > \varepsilon \right ) \rightarrow 0 \text{ as } n \rightarrow \infty,
\end{equation*}
\end{lemma}
\begin{proof}
Using the application of the general method of Bentkus from~\cite{Bentkus2003}, see Theorem~1.3 in~\cite{GotNauTikh2012}, one may reduce the problem to the Gaussian matrices and use the result of Theorem~5.1 from~\cite{Nau2012}. See also the recent paper of G{\"o}tze and Tikhomirov~\cite{GotTikh2013}.
\end{proof}

\begin{lemma}\label{l: log uniform integr}
Let $X_{jk}$, $1\le j,k\le n$, satisfy the conditions \Cond and \UI.  Then $\ln( \cdot )$ is uniformly integrable in probability with respect to $\{\nu_n\}_{n \geq 1}.$
\end{lemma}

\begin{proof}[Proof of Lemma~\ref{l: log uniform integr}]
It is enough to show that there exist $p, q > 0$ such that
\begin{equation} \label{eq:log+ u.i.}
\lim_{t \rightarrow \infty} \varlimsup_{n \rightarrow \infty} \mathbb P \left (\int_0^\infty x^{p} \nu_n(dx) > t \right ) = 0
\end{equation}
and
\begin{equation} \label{eq:log- u.i.}
\lim_{t \rightarrow \infty} \varlimsup_{n \rightarrow \infty} \mathbb P \left (\int_0^\infty x^{-q} \nu_n(dx) > t \right ) = 0.
\end{equation}
From Kolmogorov's strong law of large numbers it follows that
$$
\int_0^\infty x^2 d\mathcal{F}(x,0) \le \frac{1}{n^2} \sum_{i,j=1}^n X_{ij}^2 \rightarrow 1 \text{ as } n \rightarrow \infty.
$$
Applying this and the fact that $s_i(n^{-1/2}\X - z \I) \le s_i(n^{-1/2}\X) + |z|$ we may conclude~\eqref{eq:log+ u.i.} taking $p = 2$.

Using the uniform integrability of the second moment of $X_{jk}$ one may prove the extension of Lemma~4.3 from~\cite{Nau2012}.
\begin{lemma}\label{l:large singular values}
If the conditions of Lemma~\ref{l: log uniform integr} hold then there exist $c > 0$ and $0 < \gamma < 1$ such that a.s. for $n \gg 1$ and $n^{1-\gamma} \le i \le n-1$
$$
s_{n-i}(n^{-1/2}\X - z\I) \geq c \frac{i}{n}.
$$
\end{lemma}
We continue now the proof of Lemma~\ref{l: log uniform integr}. Denote the event $\Omega_1:= \Omega_{1,n} = \{\omega \in \Omega: s_{n-i} > c\frac{i}{n}, n^{1-\gamma} \le i \le n-1\}$.  Let us consider the event $\Omega_2:=\Omega_{2,n} = \Omega_{1} \cap \{\omega: s_n \geq n^{-B-1/2}\}$, where $B > 0$ will be chosen later. We decompose the probability from~\eqref{eq:log- u.i.} into two terms
$$
\Pb \left (\int_0^\infty x^{-q} \nu_n(dx) > t \right ) = \mathbb I_1 + \mathbb I_2,
$$
where
\begin{align*}
&\mathbb I_1: = \Pb \left (\int_0^\infty x^{-q} \nu_n(dx) > t, \Omega_{2} \right),\\
&\mathbb I_2: = \Pb \left (\int_0^\infty x^{-q} \nu_n(dx) > t,  \Omega_{2}^c \right).
\end{align*}
We may estimate $\mathbb I_2$ by
$$
\mathbb I_2 \le \Pb(s_n(\X - \sqrt n z \I) \le n^{-B}) + \Pb(\Omega_{1}^c).
$$
From Theorem~\ref{th:main} it follows that there exist $A, B > 0$ such that
\begin{equation}\label{eq: least sv}
\Pb(s_n(\X - \sqrt n z \I) \le n^{-B}) \le C n^{-A}.
\end{equation}
By Lemma~\ref{l:large singular values}
\begin{equation} \label{eq: large sv}
\varlimsup_{n \rightarrow \infty} \Pb(\Omega_{1}^c) = 0.
\end{equation}
From~\eqref{eq: least sv} and~\eqref{eq: large sv} we conclude
$$
\varlimsup_{n \rightarrow \infty} \mathbb I_2 = 0.
$$
To prove~\eqref{eq:log- u.i.} it remains to bound  $\mathbb I_1$. From Markov's inequality
$$
\mathbb I_1 \le \frac{1}{t} \E \left[\int_0^\infty x^{-q} \nu_n(dx) \mathbb I(\Omega_{2}) \right ].
$$
By the definition of $\Omega_{2}$
\begin{align*}
\E \left[\int x^{-q} \nu_n(dx) \mathbb I(\Omega_2) \right ] \le \frac{1}{n} \sum_{i=1}^{n-\lceil n^{1-\gamma} \rceil} s_{i}^{-q} + \frac{1}{n}\sum_{i=n - \lceil n^{1-\gamma} \rceil + 1}^{n} s_{i}^{-q} \\
\le 2 n^{q (B+1/2) - \gamma} + c^{-q} \frac{1}{n} \sum_{i=1}^n \left ( \frac{n}{i}\right )^q \le 2 n^{q (B+1/2)  - \gamma} + c^{-q}\int_{0}^1 s^{-q} ds.
\end{align*}
If $0 < q < \min(1, \gamma/(B+1/2))$ then the last integral is finite and we conclude the proof of the Lemma.
\end{proof}

\section{Appendix}
Let $X_1,X_2,\ldots$ be independent random variables with
\begin{equation}\label{eq: cond0}
 \E X_k=0 \text{ and } \E X_k^2=\sigma_k^2>0.
\end{equation}
We denote $\sigma^2 = \sum_{k = 1}^n \sigma_k^2$. Introduce the following quantity
$$
D(X, \lambda) = \frac{1}{\lambda^2} \int_{|X| < \lambda} x^2 dF(x) + \int_{|X| \geq \lambda } dF(x),
$$
where $F(x)$ is the distribution function of $X$. Denote by
\begin{equation}
 Q(X,\lambda)=\sup_{a\in\mathbb R}\Pb (|X-a|\le \lambda ).
\end{equation}
We shall prove here some simple results about the concentration function of sums of independent random variables.
\begin{lemma}\label{ap:l.concentration}
Assume that the condition~\eqref{eq: cond0} holds and let $S_n=\sum_{k=1}^n X_k$. Then
\begin{equation}
  Q(S_n,\lambda)\le \frac{\sqrt \lambda}{(2\sigma^2 - 8 \sum_{i = 1}^n \E X_j^2 \mathbb I(|X_j| \geq \lambda/2))^{1/2}}.
\end{equation}
\end{lemma}
\begin{proof}
According to Theorem~3 in~\cite{Petrov1975} , Chapter II, \S 2, p. 43, we have
\begin{equation}
Q(S_n,\lambda)\le A\lambda\left\{\sum_{k=1}^n\lambda_k^2 D(\widetilde X_k, \lambda_k)\right\}^{-\frac12},
\end{equation}
where $\widetilde X_k=X_k- X_k'$, $X_k'$ is independent copy of $X_k$, and $0<\lambda_k\le \lambda$.
Note that
$$
\lambda_k^2 D(\widetilde X_k, \lambda_k) \geq \E \widetilde X_k^2 \mathbb I(|\widetilde X_k| < \lambda_k) = 2\sigma_k^2 - \E \widetilde X_k^2 \mathbb I(|\widetilde X_k| \geq \lambda_k).
$$
Furthermore,
\begin{align*}
&\E \widetilde X_k^2 \mathbb I(|\widetilde X_k| \geq \lambda_k) \le 2(\E (X_k^2 + (X_k')^2) \mathbb I(|\widetilde X_k| \geq \lambda_k) \le \\
&\le 2 \E X_k^2 \mathbb I(|X_k| \geq \lambda_k/2) + 2 \E X_k^2 \mathbb I(|\widetilde X_k| \geq \lambda_k, |X_k| \le \lambda_k/2) + \\
&\le 2 \E (X_k')^2 \mathbb I(|X_k'| \geq \lambda_k/2) + 2 \E (X_k')^2 I(|\widetilde X_k| \geq \lambda_k, |X_k'| \le \lambda_k/2) \le \\
&\le 2 \E X_k^2 \mathbb I(|X_k| \geq \lambda_k/2) + 2 \E X_k^2 \mathbb I(|X_k'| \geq \lambda_k/2, |X_k| \le \lambda_k/2) +\\
&\le 2 \E (X_k')^2 \mathbb I(|X_k'| \geq \lambda_k/2) + 2 \E (X_k')^2 \mathbb I(|X_k| \geq \lambda_k/2, |X_k'| \le \lambda_k/2) \le \\
&\le 8 \E X_k^2 \mathbb I(|X_k| \geq \lambda_k/2).
\end{align*}
This implies that
$$
\lambda_k^2 D(\widetilde X_k, \lambda_k) \geq 2\sigma_k^2 - 8 \E X_k^2 \mathbb I(|X_k| \geq \lambda_k/2).
$$
We take $\lambda_k=\lambda$ and conclude the statement of the lemma.
\end{proof}

\begin{lemma}\label{l,a:reduction}
Let $S_J = \sum_{i \in J} \xi_i$, where $J \subset [n]$, and $I \subset J$ then
$$
Q(S_J, \lambda ) \le Q(S_I, \lambda).
$$
\end{lemma}
\begin{proof}
Let us fix an arbitrary $v$. From the independence of $\xi_i$ we conclude
\begin{align*}
\Pb (|S_J - v| \le \lambda ) \le \E \Pb (|S_I + S_{J / I} - v| \le \lambda |\{ \xi_i\}_{i \in I}) \le Q(S_I, \lambda).
\end{align*}
\end{proof}

\begin{lemma} \label{l,a:incomp vec}
Let $\delta, \tau \in (0,1)$. Let $x \in \incomp(\delta,\tau)$. Then there exists a set $\sigma(x) \in [n]$ of cardinality $|\sigma(x)| \geq \frac{1}{2}n\delta$ such that
$$
\sum_{k \in \sigma(x)}|x_k|^2 \geq \frac{1}{2}\tau^2
$$
and
$$
\frac{\tau}{\sqrt{2 n}} \le |x_k| \le \frac{\sqrt 2}{\sqrt{\delta n}} \text{ for any } k \in \sigma(x),
$$
which we shall call $\sigma(x)$ the "spread set of $x$".
\end{lemma}
\begin{proof}
  See in~\cite[Lemma~3.4]{RudVesh2008} and~\cite[Lemma~4.3]{GotTikh2010}.
\end{proof}

\begin{lemma}\label{ap:l_conc_funct_2}
Assume that the condition~\eqref{eq: cond0} holds and let $S_n = \sum_{k = 1}^n a_k X_k$, where $\mathbf a = (a_1, ... , a_n) \in \incomp(\delta_n,r_n)$. We additionally suppose that
\begin{equation} \label{eq: uniform integr condition app}
\max_{i=1,...,n} \E|X_{i}|^2\mathbb I{\{|X_{jk}|>M\}} \rightarrow 0 \quad \text{as}\quad M \rightarrow \infty
\end{equation}
Then there exist constants $c$ and $c'$ such that
$$
Q(S_n,\varepsilon) \le c n^{-1/4} \delta_n^{-3/8}
$$
for all $\varepsilon < c' \left (\frac{\delta_n}{n} \right)^{1/4} r_n$.
\end{lemma}
\begin{proof}
For an arbitrary set $I \in [n]$ denote $S_I:= \sum_{i \in I} a_{i} X_i$, where $a_{i}$ are the coordinates of $\mathbf a$. We shall write $S_n$ instead of $S_I$ if $I = [n]$.
From Lemma~\ref{l,a:incomp vec} there exist a set $\mathcal I$ such that $|\mathcal I| \geq \frac{1}{2} \delta_n n$ and
$$
\frac{r_n}{\sqrt{2 n}} \le |a_{i}| \le \frac{\sqrt{2}}{\sqrt{\delta_n n}} \text{ for any } i \in \mathcal I.
$$
From Lemma~\ref{l,a:reduction}
$$
Q(S_{n}, \varepsilon ) \le Q(S_{\mathcal I}, \varepsilon).
$$
Now we may split the set $\mathcal{I}$ into a union of the following disjoint sets $I_l$:
\begin{align*}
&I_l = \left \{i \in \mathcal{I}: \frac{2^{l-1} r_n}{\sqrt{2 n}} \le |a_{i}| \le \frac{2^l r_n}{\sqrt{2 n}} \right \}, l = 1, ... , L, \\
&I_{L+1} = \left \{i \in \mathcal{I}: \frac{2^{L} r_n}{\sqrt{2 n}} \le |a_{i}| \le  \frac{\sqrt{2}}{\sqrt{\delta_n n}} \right \},
\end{align*}
where $L = [c(Q) \ln n]$. It
 follows from Dirichlet's principle that there exists $l_0, 1 \le l_0 \le L+1$, such that $|I_{l_o}| \geq c \delta_n n\ln^{-1} n$. Let's set $\mathcal{J}:=I_{l_0}$.
Again by Lemma~\ref{l,a:reduction} we may write  $Q(S_{\mathcal{I}}, \varepsilon) \le Q(S_{\mathcal{J}}, \varepsilon)$.
Denote
$$
\eta_n = \frac{\sqrt{2n}}{2^{l_0 - 1} r_n}.
$$
One may check that $\eta_n \geq c \sqrt{\delta_n n} \geq c$. Let $\widetilde a_{i} = \eta_n a_{i}$ and $\widetilde S_I = \sum_{i \in I} \widetilde a_{i} X_i$. Then on the set $\mathcal J$
$$
1 \le |\widetilde a_{i}| \le 2
$$
and the variance $\sigma^2$ of the sum $\widetilde S_{\mathcal J}$ is bounded below by $|\mathcal J|$. Applying Lemma~\ref{ap:l.concentration} we get
\begin{equation}\label{eq: bound of the conc function}
Q(\widetilde S_{\mathcal J}, \varepsilon \eta_n) \le \frac{\sqrt{\varepsilon \eta_n}}{\sqrt{|\mathcal J|}(2 - 8 \max\limits_{i \in \mathcal J} \E|\widetilde a_{i} X_i|^2 \mathbb I{\{|\widetilde a_{i} X_i| \geq \varepsilon \eta_n /2\}})^{1/2}}
\end{equation}
Note that
\begin{equation*}\label{eq: u.i. secon moment}
\max_{i \in \mathcal J} \E|\widetilde a_{i} X_i|^2 I{\{|\widetilde a_{i} X_i| \geq \varepsilon \eta_n /2\}} \le 2 \max_{i = 1, ... , n} \E|X_i|^2 \mathbb I{\{|X_i| \geq \varepsilon \eta_n /4\}}.
\end{equation*}
We take
$$
\varepsilon = M \frac{(n \delta_n)^{1/4}}{\eta_n},
$$
where $M$ is some constant. It follows from~\eqref{eq: uniform integr condition app}that we may choose $M$ such that
\begin{equation}\label{eq: applying uni integr}
\max_{i = 1, ... , n} \E|X_i|^2 \mathbb I{\{|X_{jk}| \geq \varepsilon \eta_n /4\}} \le 1/16.
\end{equation}
This concludes the statement of the Lemma from~\eqref{eq: bound of the conc function} and~\eqref{eq: applying uni integr}.
\end{proof}


\bibliographystyle{plain}
\bibliography{literatur}
\end{document}